\documentclass[graybox]{svmult}

\usepackage{mathptmx}       

\DeclareMathAlphabet{\mathcal}{OMS}{cmsy}{m}{n}

\usepackage{helvet}         
\usepackage{courier}        
\usepackage{type1cm}        
\usepackage{makeidx}         
\usepackage{graphicx}        
\usepackage{multicol}        
\usepackage[bottom]{footmisc}

\usepackage{url,xspace}
\usepackage{amsfonts}

\usepackage{paralist}

\newcommand{\mon}[1]{\textnormal{Mon}(#1)}
\newcommand{\aut}[1]{\textnormal{Aut}(#1)}
\newcommand{\txn}[1]{\textnormal{#1}}
\newcommand{\rsphere}{\mathbb{C}\mathbb{P}^1}
\newcommand{\algbr}{\overline{\mathbb{Q}}}
\newcommand{\absgal}{\txn{Gal}(\algbr | \mathbb{Q})}
\newcommand{\inv}[1]{#1^{-1}}
\newcommand{\gal}[1]{\textnormal{Gal}(#1 | \mathbb Q)}
\newcommand{\per}{(\sigma,\alpha,\varphi)}
\newcommand{\interval}[1]{\left\langle #1 \right\rangle}
\newcommand*{\symmdiff}{\bigtriangleup}
\newcommand{\belyi}{Bely\u{\i}\xspace}
\newcommand{\map}{\mathcal M}

\newtheorem{thm}{Theorem}[section]
\newtheorem{cor}[thm]{Corollary}
\newtheorem{exmp}{Example}[section]
\newtheorem{prop}{Proposition}[section]
\newtheorem{conv}{Convention}[section]
\newtheorem{invar}{Invariant}[section]
\newtheorem{dfn}{Definition}[section]
\newtheorem{rmrk}{Remark}[section]
\newtheorem{qst}{Question}[section]

\begin{document}

\title*{Dessins, their delta-matroids and partial duals}
\author{Goran Mali\'{c}}
\institute{Goran Mali\'{c} \at School of Mathematics, University of Manchester, Oxford Road, Manchester M13 9PL, \email{goran.malic@manchester.ac.uk}}
%
%
\maketitle

\abstract*{Given a map $\map$ on a connected and closed orientable surface, the delta-matroid of $\map$ is a combinatorial object associated to $\map$ which captures some topological information of the embedding. We explore how delta-matroids associated to dessins behave under the action of the absolute Galois group. Twists of delta-matroids are considered as well; they correspond to the recently introduced operation of partial duality of maps. Furthermore, we prove that every map has a partial dual defined over its field of moduli. A relationship between dessins, partial duals and tropical curves arising from the cartography groups of dessins is observed as well.}

\abstract{Given a map $\map$ on a connected and closed orientable surface, the delta-matroid of $\map$ is a combinatorial object associated to $\map$ which captures some topological information of the embedding. We explore how delta-matroids associated to dessins behave under the action of the absolute Galois group. Twists of delta-matroids are considered as well; they correspond to the recently introduced operation of partial duality of maps. Furthermore, we prove that every map has a partial dual defined over its field of moduli. A relationship between dessins, partial duals and tropical curves arising from the cartography groups of dessins is observed as well.}

\section{Introduction}

A \emph{map} on a connected and orientable closed surface $X$ is a cellular embedding of a connected graph $G$ (loops and multiple edges are allowed). By this we mean that the vertices of $G$ are distinguished points of the surface, and the edges are open 1-cells drawn on the surface so that their closures meet only at the vertices; furthermore, the removal of the all the vertices and all the edges from the surface decomposes the surface into a union of open 2-cells, which are called the \emph{faces} of the map.
\begin{figure}[ht]
\sidecaption
\centering
\includegraphics[scale=.4,trim={2cm 17cm 2cm 6cm}]{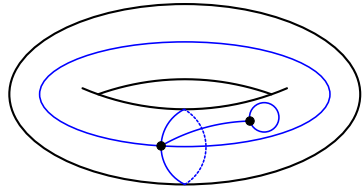}
\caption{A map with 2 vertices, 4 edges, and 2 faces on a genus 1 surface.}
\end{figure}

To every map on $X$ a \emph{clean dessin d'enfant} corresponds. A clean dessin d'enfant is a pair $(X,f)$ where $X$ is a compact Riemann surface (or, equivalently, an algebraic curve) defined over $\mathbb C$ and $f\colon X\to\rsphere$ is a holomorphic ramified covering of the Riemann sphere, ramified at most over a subset of $\{0,1,\infty\}$, with ramification orders over $1$ all equal to 2. Vertices of the map correspond to the points in the fiber above 0, whilst the preimages $\inv f(\left\langle 0,1\right\rangle)$ of the open unit interval, glued together at the fiber above 1, form the edges.

The following theorem of \belyi \cite{belyi80, belyi02} is considered as the starting point of the theory of dessins d'enfants.
\begin{theorem}[\belyi]Let $X$ be an algebraic curve defined over $\mathbb C$. Then $X$ is defined over the field $\algbr$ of algebraic numbers if, and only if there is a holomorphic ramified covering $f\colon X\to\rsphere$ of the Riemann sphere, ramified at most over a subset of $\{0,1,\infty\}$.
\end{theorem}

As a direct consequence, given any dessin $(X,f)$, both the algebraic curve $X$ and the covering map $f$ are defined over $\algbr$ and therefore the absolute Galois group $\absgal$ acts naturally on both. One of the major themes of the theory of dessin d'enfants is the identification of combinatorial, topological or geometric properties of dessins which remain invariant under the aforementioned action. We will call such invariants \emph{Galois invariants}. A number of Galois invariants have been documented and an incomplete list can be found in section \ref{subsection: invariants} of this paper or in \cite[ch. 2.4.2.2]{lando_zvonkin}.

A \emph{delta-matroid} is a combinatorial object associated to a map $\map$ on a surface $X$ which records a certain independence structure. It is completely determined by the spanning \emph{quasi-trees} of $\map$, that is the spanning sub-graphs of the underlying graph of $\map$ which can be embedded as a map with precisely one face in some surface, not necessarily the same one as $X$. We will study the behaviour of the delta-matroid of a clean dessin under the action of $\absgal$; the main conclusion is that the delta-matroid itself is not Galois invariant, however further consideration suggests that the self-dual property of delta-matroids might be, and in some cases is, preserved by the action.

A \emph{partial dual} of a map with respect to some subset of its edges is an operation which generalises the geometric dual of a map. It was recently introduced in \cite{chmutov09} and generalised to hypermaps in \cite{chmutov14}. It was shown in \cite{moffat14} that the delta-matroids of partial duals of a map $\map$ correspond to the \emph{twists} of the delta-matroid of $\map$. We give a proof of this correspondence without invoking the machinery of ribbon graphs used in \cite{moffat14} and use it to show that a map always has a partial dual defined over its field of moduli.

Towards the end of the paper we discuss the connection between maps, partial duals, and tropical curves. An \emph{abstract tropical curve} is a connected graph without vertices of degree 2 and with edges decorated by the set of positive reals and $\infty$. We associate a tropical curve to a map via the \emph{monodromy graph} of a map. The vertices of these graphs correspond to the partial duals of the map and the tropical curves obtained in this way show some similarities with maps when considering the action of $\absgal$ on them. For example, the number of vertices, edges and the genus of tropical curves remains invariant under the action of $\absgal$.

The paper is structured as follows. In section \ref{section: dessins} we define (not just clean) dessins d'enfants, describe the correspondence between dessins and bipartite maps and give a permutation representation.

In section \ref{section: Belyi} we revisit \belyi's theorem and go into more detail about the action of $\absgal$ on dessins. Some Galois invariants are described in subsection \ref{subsection: invariants} as well.

In section \ref{section: matroids} we introduce matroids and delta-matroids and describe how they arise from maps on surfaces.

In section \ref{section: Galois action on matroids} we discuss the behaviour of delta-matroids of maps when the maps are acted upon by $\absgal$. Special consideration is given to maps with self-dual delta-matroids in subsection \ref{subsection: self duality}.

In section \ref{section: partial duals} partial duals of maps are introduced, with remarks on the partial duals of hypermaps. We discuss both the combinatorial and geometric interpretation. In subsection \ref{subsection: partial duals Galois} we give a link from \cite{moffat14} between partial duals and delta-matroids and use it to show that a map always has a partial dual defined over its field of moduli.

In section \ref{section: tropical} we present a relationship between maps, their partial duals and tropical curves and note some similarities between the tropical curves associated to dessins that are in the same orbit of $\absgal$.

\section{Dessins and bipartite maps}
\label{section: dessins}

Throughout this paper $X$ shall denote a compact Riemann surface or its underlying connected and closed orientable topological surface. Furthermore, since compact Riemann surfaces are algebraic, $X$ shall denote an algebraic curve as well. We consider $X$ to be oriented, with positive orientation. Permutations shall be multiplied from left to right.

\begin{dfn}\normalfont A \emph{dessin d'enfant}, or just \emph{dessin} for short, is a pair $(X,f)$ where $X$ is a compact Riemann surface (or, equivalently, an algebraic curve) defined over $\algbr$ and $f\colon X\to\rsphere$ is a holomorphic ramified covering of the Riemann sphere, ramified at most over a subset of $\{0,1,\infty\}$.\end{dfn}

The pair $(X,f)$ is called a \emph{\belyi pair} as well, whilst the map $f$ is called a \emph{\belyi map} or a \emph{\belyi function}. Sometimes we will denote a dessin by $D=(X,f)$ to emphasise both the curve and the \belyi map. A dessin is of \emph{genus} $g$ if $X$ is of genus~$g$.

Two dessins $(X_1,f_1)$ and $(X_2,f_2)$ are isomorphic if they are isomorphic as coverings, that is if there is an orientation preserving homeomorphism $h\colon X_1 \to X_2$ such that $f_2\circ h=f_1$.

Under the terminology of Grothendieck and Schneps \cite{schneps94, schneps_lochak97vol1}, a dessin is called \emph{pre-clean} if the ramification orders above $1$ are at most 2, and \emph{clean} if they all are precisely equal to 2. The associated \belyi maps are called \emph{pre-clean} and \emph{clean \belyi maps}, respectively.

\begin{dfn}\normalfont A bipartite map on $X$ is a map on a topological surface $X$ with bipartite structure, that is the set of vertices can be decomposed into a disjoint union $B\cup W$ such that every edge is incident with precisely one vertex from $B$ and one vertex from $W$. Vertices from $B$ and $W$ are called black and white, respectively.\end{dfn}

Two bipartite maps $\map_1$ on $X_1$ and $\map_2$ on $X_2$ are isomorphic if there is an orientation preserving homeomorphism $X_1\to X_2$ which restricts to a bipartite graph isomorphism. When working with bipartite maps we shall adopt the following.

\begin{conv}\normalfont  The segments incident with precisely one black and one white vertex in a bipartite map shall be called \emph{darts}. Since every map can be thought of as a bipartite map by considering the edge midpoints as white vertices (see figure \ref{figure3}), we shall reserve the term \emph{edge} for maps only. To summarise, a bipartite map has darts, not edges, whilst an edge of a map has precisely two darts.\end{conv}

To every bipartite map on a topological surface $X$ a dessin corresponds, and vice-versa. This correspondence is realised in the following way: given a dessin $(X,f)$, the preimage $f^{-1}([0,1])$ of the closed unit interval will produce a bipartite map on the underlying surface of the curve $X$ such that the vertices of the map correspond to the points in the preimages of $0$ and $1$, and the darts correspond to the preimages of the open unit interval. The bipartite structure is obtained by colouring the preimages of $0$ in black and the preimages of $1$ in white.

On the other hand, given a bipartite map on a topological surface $X$, colour the vertices in black and white so that the bipartite structure is respected. To the interior of each face add a single new vertex and represent it with a diamond $\diamond$, so that it is distinguished from the black and white vertices. Now triangulate $X$ by connecting the diamonds with the black and white vertices that are on the boundaries of the corresponding faces. Following the orientation of $X$, call the triangles with vertices oriented as $\bullet$-$\circ$-$\diamond$-$\bullet$ positive, and call other triangles negative (see figure \ref{figure2}).
\begin{figure}[ht]
  \centering
  \includegraphics[scale=.7, trim={5cm 17cm 3.5cm 5.3cm}]{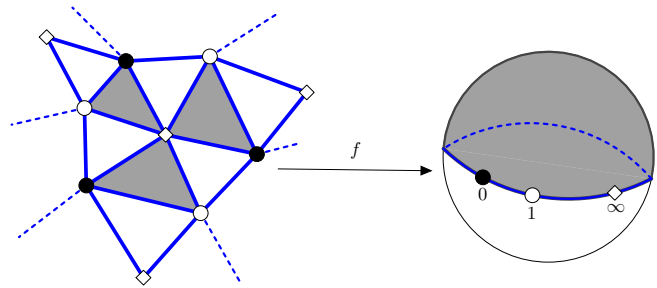}
  \caption{The positive (shaded) and negative triangles are mapped to the upper and lower-half plane, respectively. The sides of the triangles are mapped to $\mathbb R\cup\{\infty\}$ so that the black and white vertices map to 0 and 1, respectively, and the face centres map to $\infty$.}\label{figure2}
\end{figure}
Now map the positive and negative triangles to the upper and lower half-plane of $\mathbb C$, respectively, and map the sides of the triangles to the real line so that the black, white and diamond vertices are mapped to $0$, $1$ and $\infty$, respectively. As a result, a ramified cover $f\colon X\to\rsphere$, ramified only over a subset of $\{0,1,\infty\}$ will be produced. We now impose on $X$ the unique Riemann surface structure which makes $f$ holomorphic. For a detailed description of this correspondence see \cite[sections 4.2 and 4.3]{girondo_gonzalez-diez}.

\begin{rmrk}\normalfont  In the introduction we stated that maps correspond to clean dessins. Here we explain why this is the case: a given map with $n$ edges can be refined into a bipartite map $2n$ darts by adding the edge midpoints of the map as white vertices. The corresponding \belyi function will obviously have ramification orders at the white vertices equal to 2. In the other way, given a clean dessin, we first obtain a bipartite map with $2n$ darts in which every white vertex is incident to precisely two darts, since all the ramification orders above 1 are equal to 2. By ignoring the white vertices we obtain a map with $n$ edges. See figure \ref{figure3} for an example.
\begin{figure}[ht]
  \centering

  \includegraphics[scale=.85,trim={6.5cm 18cm 6.25cm 5cm}]{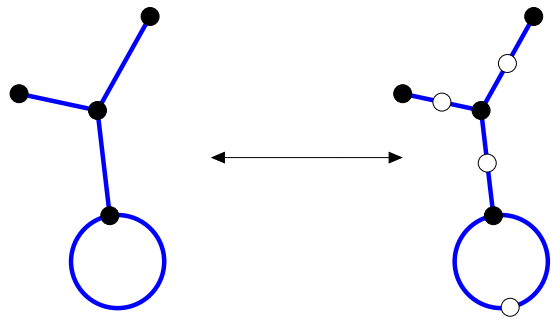}
  \caption{A map (left) is transformed into a clean dessin (right) by adding edge midpoints as white vertices. In the other way, from a clean dessin we obtain a map by ignoring the white vertices.}\label{figure3}
\end{figure}
\end{rmrk}

From now on we shall think of dessins both as bipartite maps, and as \belyi pairs. Consequently, clean dessins are synonymous with maps.

\subsection{A permutation representation of dessins}\label{subsection:perm_rep}

Throughout this section let $(X,f)$ be a dessin with $n$ darts (or, equivalently, such that $f$ is a degree $n$ ramified covering). The goal of this section is to describe how each such dessin can be represented by a triple $\per$ of permutations in $S_{n}$. However, let us first introduce the following labelling convention to which we will conform throughout the rest of this paper.

\begin{conv}\normalfont We label the darts of a dessin with the elements of the set $\{1,\dots,n\}$ so that, when standing at a black vertex, and looking towards an adjacent white vertex, the label is placed on the `left side' of the dart. See figure \ref{figure4} for an example.
\end{conv}
\begin{figure}[ht]
  \sidecaption[t]
  \centering
  \includegraphics[scale=.7, trim={5.5cm 20.25cm 5.5cm 4.75cm}]{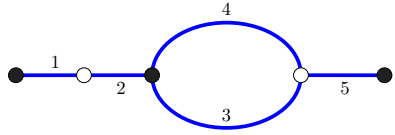}
  \caption{Labelling of darts. The labels are always on the left when looking from a black vertex to its adjacent white vertices.}\label{figure4}
\end{figure}

Following the previous convention, label the darts of a dessin arbitrarily. Now let $\sigma$ and $\alpha$ denote the permutations which record the cyclic orderings of the labels around the black and white vertices, respectively, and let $\varphi$ denote the permutation which records the counter-clockwise ordering of the labels within each face.
\begin{exmp}\normalfont For the dessin in figure \ref{figure4} we have $\sigma=(1)(2\,3\,4)(5)$, $\alpha=(1\,2)(3\,5\,4)$ and $\varphi=(1\,4\,5\,2)(3)$. The cycles of length 1 are usually dropped from the notation. Note that the cycle corresponding to the `outer face' is, from the reader's perspective, recorder clockwise. This does not violate our convention since that face should be viewed from the opposite side of the sphere \cite[remark 1.3.18(3)]{lando_zvonkin}.
\end{exmp}

Since the labelling was arbitrary, a change of labels corresponds to simultaneous conjugation of $\sigma$, $\alpha$ and $\varphi$ by some element in $S_n$. Therefore, any dessin can be represented, up to conjugation, as a triple of permutations.

\begin{dfn}\normalfont The length of a cycle in $\sigma$ or $\alpha$ corresponding to a black or a white vertex, respectively, is called the \emph{degree} of the vertex. The length of a cycle in $\varphi$ corresponding to a face is called the \emph{degree} of the face. Thus, the degree of a vertex is the number of darts incident to it, while the degree of a face is half the number of darts on its boundary.
\end{dfn}

A triple $(\sigma,\alpha,\varphi)$ representing a dessin $D=(X,f)$ satisfies the following properties:
\begin{itemize}
  \item the group $\left\langle\sigma,\alpha,\varphi\right\rangle$ acts transitively on the set $\{1,\dots,n\}$ and
  \item $\sigma\alpha\varphi=1$.
\end{itemize}
The first property above is due to the fact that dessins are connected while the second is due to the following: consider three non-trivial simple loops $\gamma_0$, $\gamma_1$ and $\gamma_\infty$ on $\rsphere\setminus\{0, 1, \infty\}$ based at $1/2$ and going around 0, 1 and $\infty$ once, respectively. The lifts of these loops under $f$ correspond to paths on $X$ that start at some and end at another (possibly the same) point in $\inv f(1/2)$ . We observe the following.
\begin{itemize}
  \item Every dart of $D$ contains precisely one element of $\inv f(1/2)$ since $f$ is unramified at $1/2$.
  \item The cardinality of $\inv f(1/2)$ is precisely $n$. Hence there is a bijection between $\inv f(1/2)$ and $\{1,\dots,n\}$.
  \item With respect to this bijection, $\sigma$, $\alpha$ and $\varphi$  can be thought of as permutations of the set $\inv f(1/2)$.
\end{itemize}
Therefore the loops $\gamma_0$, $\gamma_1$ and $\gamma_\infty$ induce $\sigma$, $\alpha$ and $\varphi$. Since the product $\gamma_0\gamma_1\gamma_\infty$ is trivial, the corresponding permutation $\sigma\alpha\varphi$ must be trivial as well.

We have now seen that to every dessin with $n$ darts we can assign a triple of permutations in $S_n$ such that their product is trivial and the group that they generate acts transitively on the set $\{1,\dots, n\}$. In a similar fashion we can show that this assignment works in the opposite direction: given three permutations $\sigma$, $\alpha$ and $\varphi$ in $S_n$ such that $\sigma\alpha\varphi=1$ and the group that they generate acts transitively on $\{1,\dots,n\}$, we can construct a dessin with $n$ darts so that the cyclic orderings of labels around vertices correspond to the cycles of $\sigma$, $\alpha$ and $\varphi$, up to simultaneous conjugation. Therefore, up to simultaneous conjugation, a dessin is uniquely represented by a transitive triple $(\sigma,\alpha,\varphi)$ with $\sigma\alpha\varphi=1$, and such a triple recovers a unique dessin (up to isomorphism).

\begin{rmrk}\normalfont Obviously, dessins correspond to $2$-generated transitive permutation groups since we can set $\varphi=\inv{(\sigma\alpha)}$. However, we prefer to emphasise all three permutations.
\end{rmrk}

We shall use the notation $D=\per$ to denote that a dessin $D$ is represented by the triple $\per$.

\begin{dfn}\normalfont The subgroup of $S_n$ generated by $\sigma$, $\alpha$ and $\varphi$ is called \emph{the monodromy group} of $D=\per$ and denoted by $\mon D$.\end{dfn}

The monodromy group is actually defined up to conjugation in order to account for all the possible ways in which a dessin can be labelled.

\begin{exmp}\normalfont The monodromy group of the dessin in figure \ref{figure3} is (isomorphic to) $\txn{PSL}_3(2)$. The monodromy group of the dessin in figure \ref{figure4} is $S_5$.
\end{exmp}

\section{\belyi's theorem and the Galois action on dessins}\label{section: Belyi}

One of the most mysterious objects in mathematics is the absolute Galois group $\absgal$, the group of automorphisms of $\algbr$ that fix $\mathbb Q$ point-wise, and the study of its structure is one of the goals of the Langlands program. Grothendieck, in his remarkable \emph{Esquisse d'un Programme} \cite{grothendieck97}, envisioned an approach towards understanding $\absgal$ as an automorphism group of a certain topological object;  the starting point of his approach is \belyi's theorem, which we restate here.

\begin{thm}[\belyi]\label{belyi}Let $X$ be an algebraic curve defined over $\mathbb C$. Then $X$ is defined over $\algbr$ if, and only if there is a holomorphic ramified covering $f\colon X\to\rsphere$, ramified at most over a subset of $\{0,1,\infty\}$.
\end{thm}

Aside from \belyi's own papers \cite{belyi80, belyi02}, various other proofs can be found in, for example, \cite[theorem 4.7.6]{szamuely09} or \cite[chapter 3]{girondo_gonzalez-diez} or the recent new proof in \cite{goldring14}. \belyi himself concluded that the above theorem implies that $\absgal$ embeds into the outer automorphism group of the profinite completion of the fundamental group of $\rsphere\setminus\{0,1,\infty\}$, however it was Grothendieck who observed that $\absgal$ must therefore act faithfully on the set of dessins as well. This interplay between algebraic, combinatorial and topological objects is what prompted Grothendieck to develop his \emph{Esquisse}. For more detail, see \cite{schneps_lochak97vol1} or \cite{szamuely09}.

\subsection{Galois action on dessins}\label{subsection: Galois action}

Let $D=(X,f)$ be a dessin. If $X$ is of genus 0, then necessarily $X=\rsphere$ and $f\colon\rsphere\to\rsphere$ is a rational map with critical values in the set $\{0,1,\infty\}$. If $f=p/q$, where $p,q \in\mathbb C[z]$, then \belyi's theorem implies that $p,q\in\algbr[z]$. Moreover, the coefficients of both $p$ and $q$ generate a finite Galois extension $K$ of $\mathbb Q$. Therefore $p,q\in K[z]$, and $\gal K$ acts on $f$ by acting on the coefficients of $p$ and $q$, that is if $\theta\in\gal K$ and
\begin{eqnarray*}
  f(z) &=& \frac{a_0+a_1z+\cdots+a_mz^m}{b_0+b_1z+\cdots b_nz^n},\\[.3cm]
 \txn{then }f^\theta(z) &=& \frac{\theta(a_0)+\theta(a_1)z+\cdots+\theta(a_m)z^m}{\theta(b_0)+\theta(b_1)z+\cdots \theta(b_n)z^n}.
\end{eqnarray*}

If $X$ is of genus 1 or 2, then as an hyperelliptic algebraic curve it is defined by the zero-set of an irreducible polynomial $F$ in $\mathbb C[x,y]$. This time we must take into consideration the coefficients of both $F$ and $f$ which, due to \belyi's theorem again, generate a finite Galois extension $K$ of $\mathbb Q$. Similarly as in the genus 0 case, $\gal K$ acts on $D$ by acting on the coefficients of both $F$ and $f$ simultaneously. When the genus of $X$ is at least 3, the action is exhibited similarly.

It is not immediately clear that the action of some automorphism in $\gal K$ on a \belyi map $f$ will produce a \belyi map. This indeed is the case and we refer the reader to the discussion in \cite[ch. 2.4.2]{lando_zvonkin}.

Since any $\mathbb Q$-automorphism of $K$ extends to an $\mathbb Q$-automorphism of $\algbr$ \cite[ch. 3]{bor_jan}, we truly have an action of $\absgal$ on the set of dessins.

We shall denote by $D^\theta=(X^\theta,f^\theta)$ the dessin that is the result of the action of $\theta\in\absgal$ on $D=(X,f)$. We shall also say that $D^\theta$ is \emph{conjugate} to $D$.

The following example is borrowed from \cite[ex. 2.3.3]{lando_zvonkin}.

\begin{exmp}\normalfont \label{exmp1}Let $D=(X,f)$ be a dessin where $X$ is the elliptic curve
\[y^2=x(x-1)(x-(3+2\sqrt 3)),\]
and $f\colon X\to \rsphere$ is the composition $g\circ \pi_x$, where $\pi_x \colon X\to\rsphere$ is the projection to the first coordinate and $g\colon\rsphere\to\rsphere$ is given by
\[g(z)=-\frac{(z-1)^3(z-9)}{64z}.\]
The corresponding bipartite map is depicted on the left in figure \ref{figure5}.

\begin{figure}[ht]
  \centering
    \includegraphics[trim=7cm 14cm 7cm 5cm, width=5cm]{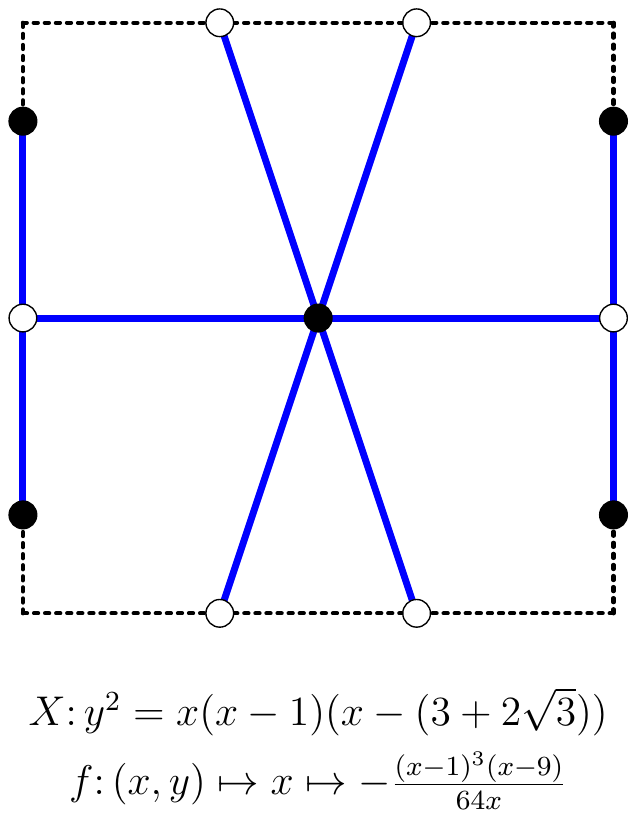}
    \includegraphics[trim=7cm 14cm 7cm 5cm, width=5cm]{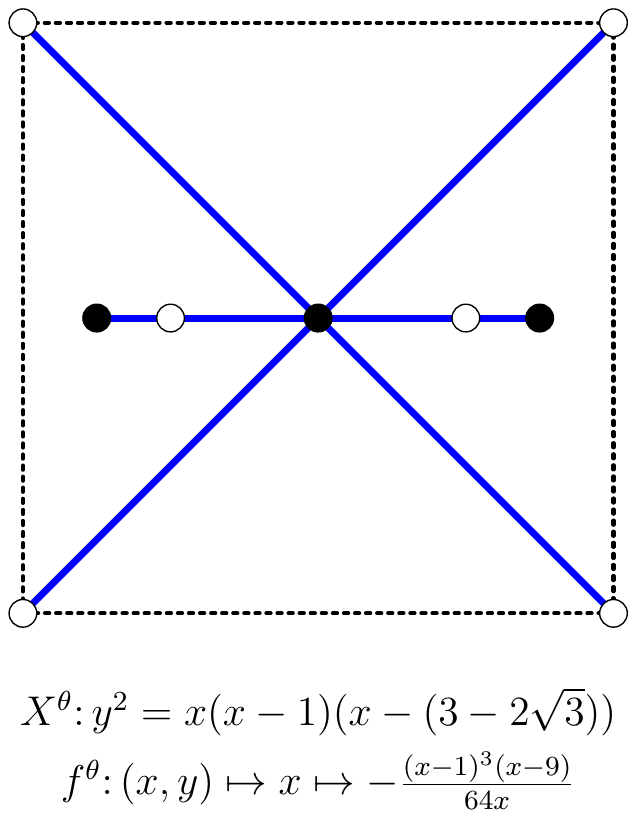}
  \caption{The two dessins $(X,f)$ and $(X^\theta,f^\theta)$ from example \ref{exmp1}. The dotted lines indicate the boundary of the polygon representation of an orientable genus 1 surface with the usual identification of the left-and-right and top-and-bottom sides.}\label{figure5}
\end{figure}

Note that we must consider $g\circ \pi_x$ and not just $\pi_x$ since $\pi_x$ is not a \belyi map; it is ramified over four points, namely $0$, $1$, $3+2\sqrt 3$ and $\infty$. However, $g$ maps these four points onto the set $\{0,1,\infty\}$ and therefore $g\circ\pi_x$ is a true \belyi map.

The Galois extension that the coefficients of $X$ and $f$ generate is $K=\mathbb Q(\sqrt 3)$ and the corresponding Galois group has only one non-trivial automorphism $\theta$ given by $\theta\colon\sqrt3\mapsto-\sqrt3$. Therefore $X^\theta$ is the elliptic curve $y^2=x(x-1)(x-(3-2\sqrt 3))$. The curve $X^\theta$ is non-isomorphic to $X$, which can easily be seen by computing the $j$-invariants of both.

What about $f^\theta$? In this case, $\pi_x\colon X^\theta\to\rsphere$ is unramified over $3+2\sqrt 3$ and ramified over $3-2\sqrt 3$. However, $g$ maps $3-2\sqrt 3$ to 0 as well, and since $g$ is defined over $\mathbb Q$, the \belyi functions $f$ and $f^\theta$ coincide. The bipartite map corresponding to $(X^\theta,f^\theta)$ is depicted on the right in figure \ref{figure5}.
\end{exmp}

This action of $\absgal$ on dessins is faithful already on the set of \emph{trees}, i.e.\ the genus 0 dessins with precisely one face and with polynomials as \belyi functions. However, this is not straight-forward (proofs can be found in \cite{schneps94, girondo_gonzalez-diez}) and, surprisingly, it is much easier to show faithfulness in genus 1 \cite[ch. 4.5.2]{girondo_gonzalez-diez}. Moreover, the action is faithful in every genus \cite[ch. 4.5.2]{girondo_gonzalez-diez}.

\subsection{Galois invariants}\label{subsection: invariants}

Here we shall list a number of properties of dessins which, up to various notions of equivalence, remain invariant under the action of $\absgal$. Such properties are called \emph{Galois invariants} of dessins. We shall use the notation $D\simeq D'$ to indicate that two dessins $D$ and $D'$ are conjugate.

\begin{invar}[Passport]\normalfont\label{passport}Let $D=(\sigma,\alpha,\varphi)$ be a dessin with $n$ darts. The cycle types of $\sigma$, $\alpha$ and $\varphi$ define three partitions $\lambda_\sigma$, $\lambda_\alpha$ and $\lambda_\varphi$ of $n$. The \emph{passport} of $D$ is the sequence $[\lambda_\sigma,\lambda_\alpha,\lambda_\varphi]$.
If $D'=(\sigma',\alpha',\varphi')$ and $D\simeq D'$, then $[\lambda_\sigma,\lambda_\alpha,\lambda_\varphi]=[\lambda_{\sigma'},\lambda_{\alpha'},\lambda_{\varphi'}]$. In other words, conjugate dessins have the same passport.
\end{invar}

We compactly record a partition of, for example, $n=17=3+3+3+3+2+1+1+1$ as $3^421^3$. If a double-digit number appears in the partition, for example $23=11+11+1$, then we record it as $(11)^21$.

\begin{exmp}\normalfont The dessin in figure \ref{figure3} has the sequence $[3^21^2,2^4,71]$ as its passport. The dessin in figure \ref{figure4} has $[31^2,32,41]$ as its passport. The two dessins in figure \ref{figure5} both have $[61^2,42^2,62]$ as their passport.
\end{exmp}

The passport is a very crude invariant, however much useful information can be extracted from it. For example, the number of black vertices, white vertices, darts and faces is invariant and hence the genus of the surface must also be invariant. Moreover, we can conclude that every orbit of the action is finite since there are only finitely many dessins with a given passport.

\begin{invar}[Monodromy group]\normalfont\label{monodromy}If $D\simeq D'$, then $\mon D\cong\mon{D'}$. In other words, conjugate dessins have isomorphic monodromy groups.
\end{invar}

\begin{exmp}\normalfont The monodromy group of the dessin $D$ on the left side in figure \ref{figure5} is the nilpotent group given by the external wreath product of $\mathbb Z_2$ by the alternating group $A_4$. Since the dessin on the right side of the same figure is conjugate to $D$, its monodromy group is isomorphic to $\mon D$.
\end{exmp}

The monodromy group is a much finer invariant than the passport since dessins with the same passport may have non-isomorphic monodromy groups.

\begin{invar}[Automorphism group]\normalfont Let $D=\per$. The centre of $\mon D$ in $S_n$ is the \emph{automorphism group} of $D$, denoted by $\aut D$. If $D\simeq D'$, then \mbox{$\aut D\cong\aut{D'}$}.
\end{invar}

If the automorphism group of a dessin $D$ acts transitively on the set $\{1,\dots,n\}$ or, equivalently, if $|\aut D|=n$,  then we say that the dessin is \emph{regular}. It has been shown in \cite{gonzalez-diez_jaik13, guillot14} that $\absgal$ acts faithfully on the set of regular dessins as well.

\begin{invar}[Cartography group]\normalfont The \emph{cartography group} $\txn{Cart}(D)$ of a dessin $D$ is the monodromy group of the map obtained from $D$ by colouring all the white vertices black and adding new white vertices to the midpoints of edges. Therefore, for maps or clean dessins we have $\txn{Cart}(D)=\mon D$. As it was the case with the monodromy group, conjugate dessins have isomorphic cartography groups.
\end{invar}

Since the cartography groups are subgroups of $S_{2n}$, when $n$ is large they are in general more difficult to compute than the monodromy groups. However, G. Jones and M. Streit have shown in \cite{jones_streit97} that the cartography group can be used to distinguish between the orbits of $\absgal$ when the monodromy group does not suffice. That is, there are non-conjugate dessins with isomorphic monodromy groups but non-isomorphic cartography groups.

More Galois invariant groups that arise from the monodromy group in a similar fashion can be found in \cite{matchett_wood06}.

\begin{invar}[Duality]\normalfont\label{duality}Given a dessin $D=(X,f)$ we define its \emph{dual} dessin $D^*$ to be the dessin corresponding to the \belyi pair $(X,1/f)$. Clearly, if $D_1\simeq D_2$, then $D_1^*\simeq D_2^*$.\end{invar}

In terms of permutation representations, if $D=\per$, then $D^*$ will have the triple $(\inv\varphi,\inv\alpha,\inv\sigma)$ as its permutation representation. Geometrically this means that the black vertices and the face centres of the dual are the face centres and the black vertices of $D$, respectively, while the white vertices remain unchanged, except for the orientation of the labels. The darts of $D^*$ are the curved segments that connect the face centres and the white vertices of $D$. See figure \ref{figure6} for an example.

\begin{figure}[ht]
  \sidecaption
  \centering
  \includegraphics[scale=.7,trim={5.25cm 17.15cm 5.75cm 5.5cm}]{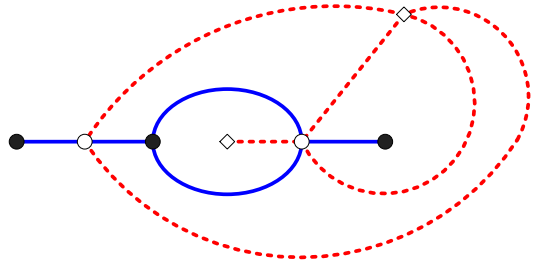}
  \caption{The dessin (full) from figure \ref{figure4} and its dual (dashed).}\label{figure6}
\end{figure}

\begin{rmrk}\normalfont If $D$ is a map then $D^*$ corresponds to the geometric dual of a map. If $e$ is an edge of $D$, then the unique edge $e^*$ in $D^*$ which intersects $e$ at the appropriate white vertex is called the \emph{coedge} of $e$.\end{rmrk}

\begin{invar}[Self-duality]\normalfont We say that a dessin is self-dual if it is isomorphic to its dual. If $D$ is self dual and $D\simeq D'$, then $D'$ is self-dual as well. We shall considered self-duality again in section \ref{subsection: self duality}.
\end{invar}

\begin{invar}[Field of moduli]\normalfont Let $D$ be a dessin and
\[\txn{Stab}(D)=\{\theta\in\absgal\mid D^\theta= D \}\]
the stabiliser of $D$ in $\absgal$. The \emph{field of moduli} of $D$ is the fixed field corresponding to $\txn{Stab}(D)$, that is the field \[\{q\in\algbr\mid\theta(q)=q, \txn{ for all }\theta\in\txn{Stab}(D)\}.\]

Alternatively, the field of moduli of $D$ is the intersection of all \emph{fields of definition} of $D$, i.e. all the fields in which we can write down a \belyi pair for $D$.
\end{invar}

Fields of moduli are notoriously difficult to compute, and moreover, there are dessins whose \belyi pairs cannot be realised over their own fields of moduli! \cite[e.g. 2.4.8 and 2.4.9]{lando_zvonkin} Therefore, a natural question to ask is when can a dessin be defined over its field of moduli. Based on the work of Birch in \cite{birch1994noncongruence} (see also \cite{SijslingVoight}), a necessary, but not sufficient condition was given in \cite{wolfart97}\footnote{See also theorem 2.4.14 in \cite{lando_zvonkin}.}.

\begin{thm}\label{theorem: bachelor}A dessin can be defined over its field of moduli if there exists a black vertex, or a white vertex, or a face center which is unique for its type and degree.
\end{thm}

\section{Matroids and delta-matroids}\label{section: matroids}

It is often said that matroids are a combinatorial abstraction of linear independence. Formally we have
\begin{dfn}\normalfont Given a non-empty finite set $E$, a \emph{matroid on} $E$ is a non-empty family $M(E)$ of subsets of $E$ which is closed under taking subsets, i.e.
\begin{itemize}
  \item if $J\in M(E)$ and $I\subseteq J$, then $I\in M(E)$,
\end{itemize}
and satisfies the following \emph{augmentation axiom}:
\begin{itemize}
  \item if $I,J\in M(E)$ with $|I|<|J|$, then there exists $x \in J\setminus I$ such that $I\cup\{x\}\in M(E)$.
\end{itemize}
The elements of $M(E)$ obviously mimic the properties of linearly independent sets of vectors and are hence called \emph{independent sets}. Subsets of $E$ which are not independent are called \emph{dependent}. Maximal independent sets are called \emph{bases}, and, as the reader might suspect, any two bases of $M(E)$ are of the same size \cite[lemma 1.2.1]{oxley92}. Two matroids $M(E)$ and $M(E')$ are isomorphic if there is a bijection $\psi\colon E\to E'$ such that $\psi(I)$ is independent if, and only if $I$ is independent.
\end{dfn}

Matroids were introduced by Hassler Whitney \cite{whitney35} and, as the name suggests, arise naturally from matrices; the collection of linearly independent sets of columns in a matrix forms a matroid \cite[prop. 1.1.1]{oxley92}. Matroids which are isomorphic to matroids arising from matrices are called \emph{representable}.

A multitude of examples of matroids arise from graphs as well. Given an abstract undirected graph $G=(V,E)$, the collection of its acyclic sets of edges forms a matroid $M(G)$ \cite[theorem 4.1]{gordon_mcnulty}. The independent sets of this matroid are in fact subsets of $E$, however we denote it by $M(G)$ to emphasise that the matroid is arising from a graph. The spanning forests of $G$ correspond to the bases of $M(G)$. If $G$ is connected then the trees and the spanning trees correspond to the independent sets and the bases of $M(G)$. Matroids which are isomorphic to matroids arising from graphs are called \emph{graphic}. Moreover, every graphic matroid is isomorphic to the graphic matroid of some connected graph \cite[prop. 1.2.8]{oxley92}.

\begin{conv}\normalfont It is customary in matroid theory to drop the braces and commas when specifying sets. For example, $abc$ stands for the set $\{a,b,c\}$.\end{conv}

Given a matroid $M(E)$ we can completely recover the independent sets by describing only the collection $\mathcal B$ of its bases. On the other hand, if $\mathcal B$ is a non-empty collection of subsets of some non-empty set $E$, then $\mathcal B$ will be the collection of bases of a matroid if, and only if the following \emph{exchange axiom} is satisfied \cite[cor. 1.2.5]{oxley92}:
\begin{itemize}
  \item if $B_1$, $B_2\in\mathcal B$ and $x\in B_1\setminus B_2$, then there is $y\in B_2\setminus B_1$ such that $(B_1\setminus x)\cup y\in\mathcal B$.
\end{itemize}

Let us look at a simple example of a graphic matroid.

\begin{exmp}\normalfont \label{exmp1matroid}Let $G$ be the map obtained from the bipartite map in figure \ref{figure4} by colouring all the white vertices into black vertices (see figure \ref{figure7}). The bases of $M(G)$ are the sets $1235$ and $1245$ and they correspond precisely to the spanning trees of the map.
\end{exmp}


Let $\mathcal B$ be the collection of bases of some matroid $M(E)$ and let
\[\mathcal B^*=\{E\setminus B\mid B\in\mathcal B\}\]
be the collection of the complements of its bases. This collection is clearly non-empty and it can be shown that it satisfies the exchange axiom \cite[ch. 2]{oxley92}. The matroid with $\mathcal B^*$ as its collection of bases is called the \emph{dual matroid} of $M(E)$, and is denoted by $M^*(E)$.

\begin{exmp}\normalfont Let us go back to the map in figure \ref{figure7}. As we have seen in example \ref{exmp1matroid}, the bases of this map are $1235$ and $1245$. Recall that the unique edge of the dual map which intersects an edge $e$ of the map is labelled by $e^*$. Therefore the bases of the dual map should be the coedges $4^*$ and $3^*$. In figure \ref{figure7} we can see that this indeed is the case.
\begin{figure}[ht]
  \centering
  \includegraphics[scale=.7,trim={5.5cm 16cm 2.5cm 5.65cm}]{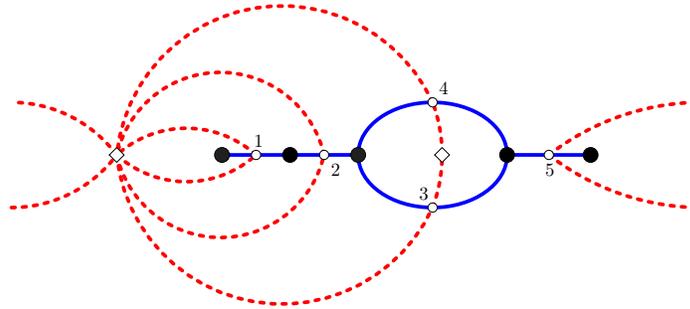}
  \caption{A map obtained from the dessin in figure \ref{figure4} by colouring the white vertices into black and adding new white vertices at the edge midpoints. The dual map (dashed) is formed by connecting the face centres to the (new) white vertices. The segments on the left and right go around the sphere and connect into a loop.}\label{figure7}
\end{figure}
\end{exmp}

We say that a matroid is \emph{cographic} if it is isomorphic to the dual of some graphic matroid. The following theorem of Whitney \cite{whitney33} establishes a matroidal characterisation of planarity.
\begin{thm}[Whitney's planarity criterion]\label{theorem: whitney}Let $G$ be a connected graph. Then $G$ is planar if, and only if $M(G)$ is cographic. Moreover, if $G$ is a plane map, then $M^*(G)=M(G^*)$, where $G^*$ is the geometric dual of $G$.\end{thm}

\subsection{Delta-matroids}\label{subsection: delta-matroids}

As we have seen in theorem \ref{theorem: whitney}, the dual matroid of a plane map is the matroid of the dual map. This correspondence does not hold for graphs that are not planar. However, we would like to extend this property to non-planar graphs and their cellular embeddings, that is to maps on surfaces of any genus. To that effect, we introduce the following.

\begin{dfn}\normalfont \label{definition: delta-matroid}A \emph{delta-matroid} $\Delta(E)$ \emph{on} $E=\{1,\dots,n\}$ is a non-empty collection $\mathcal F$ of subsets of $E$ satisfying the following \emph{symmetric axiom}:
\begin{itemize}
  \item if $F_1$, $F_2\in \mathcal F$ and $x\in F_1\symmdiff F_2$, then there is $y\in F_2\symmdiff F_1$ such that $F_1 \symmdiff \{x,y\}\in\mathcal F$.
\end{itemize}
Here $\symmdiff$ denotes the symmetric difference of sets. The elements of $\mathcal F$ are called \emph{feasible sets}. Two delta-matroids $\Delta(E)$ and $\Delta(E')$ are isomorphic if there is a bijection $\psi\colon E\to E'$ preserving feasible sets. We shall use the notation $\Delta(E)\cong \Delta(E')$ to indicate that $\Delta(E)$ and  $\Delta(E')$ are isomorphic delta-matroids.
\end{dfn}

It is straightforward to show that every matroid is a delta-matroid, however not every delta-matroid is a matroid, as we shall see.

Delta-matroids, also known as symmetric or Lagrangian matroids \cite[ch. 4]{borovik_gelfand_white}, were first introduced by Bouchet \cite{bouchet87} and later generalized to the so-called \emph{Coxeter matroids} by Gelfand and Serganova \cite{gelfand_serganova87, gelfand_serganova87b}. A systematic treatment of Coxeter matroid theory can be found in \cite{borovik_gelfand_white}.

Delta-matroids arise from maps in a fashion similar to which graphic matroids arise from graphs. However, instead of spanning trees we shall consider \emph{bases} of maps. To that effect, let $\map$ be a map on $X$ with $n$ edges labelled by the set $E=\{1,2,\dots,n\}$. Label the edges of the dual map $\map^*$ by the set $E^*=\{1^*,2^*,\dots,n^*\}$ so that $j^*$ is the coedge corresponding to $j$. Call an $n$-subset $B$ of $E\cup E^*$ admissible if precisely one of $j$ or $j^*$ appears in it.
\begin{dfn}\normalfont An admissible $n$-subset $B$ of $E\cup E^*$ is called a \emph{base} if $X\setminus B$ is connected.\end{dfn}

It was shown in \cite[proposition 2.1]{bouchet89} that the bases of $\map$ are equicardinal and spanning, that is each base includes a spanning tree of the underlying graph of $\map$.

\begin{dfn}\normalfont A \emph{quasi-tree} is a map with precisely one face. A \emph{spanning quasi-tree of a map} $\map$ is a quasi-tree obtained from a base $B$ of $\map$ by ignoring the starred elements.\end{dfn}

\begin{rmrk}\normalfont We are allowing the case of an empty spanning quasi-tree. This occurs precisely when there is a base $B=E^*$. In that case, $X\setminus E^*$ is connected and therefore $\map^*$ has precisely one face. Hence $\map$ has only one vertex and we think of the empty spanning quasi-tree as the degenerate map on the sphere with one vertex and no edges.\end{rmrk}


Let $\mathcal B$ denote the collection of bases of a map $\map$, and let $\mathcal F$ denote the collection of the spanning quasi-trees of $\map$, that is the collection
\[\mathcal F=\{E\cap B\mid B\in\mathcal B\}.\]
Analogously to matroids, the spanning quasi-trees of a map form a delta-matroid \cite[th. 4.3.1]{borovik_gelfand_white}.
\begin{thm}If $\map$ is a map on $X$, then $\mathcal F$ is the collection of feasible sets of a delta-matroid.\end{thm}

The delta-matroid arising from a map $\map$ shall be denoted by $\Delta(\map)$ or $\Delta(D)$ when we are assuming that $D$ is a clean dessin.

\begin{exmp}\normalfont Let $\map$ be a map on a genus 1 surface $X$ with two vertices, three edges and one face, as shown and labelled in figure \ref{figure8}. Since the map itself has precisely one face, then $X\setminus \map$ must be connected. Therefore 123 is a base. It is easy to see that no 2-subset of 123, together with an appropriate coedge, is a base. The remaining admissible $3$-sets are $12^*3^*$, $1^*23^*$, $1^*2^*3$ and $1^*2^*3^*$. Out of those four, only $12^*3^*$ and $1^*2^*3$ do not disconnect $X$. Therefore, the feasible sets are $123$, $1$, $3$.
\begin{figure}[ht]
  \centering
\includegraphics[scale=.75,trim={7cm 16.2cm 7cm 5cm}]{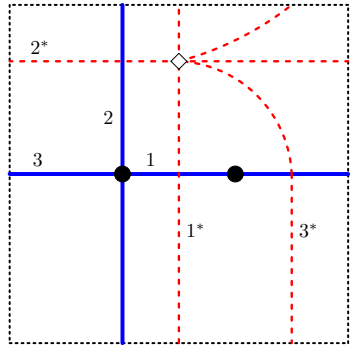}
  \caption{The bases of the map are 123, $12^*3^*$ and $1^*2^*3$. Hence $\Delta(\map)=\{123, 1,3\}$. The edges 1 and 3 are the spanning quasi-trees of $\map$ which can be embedded as maps only on the sphere.}\label{figure8}
\end{figure}

In general one does not need to go through all possible admissible $n$-subsets of $E\cup E^*$ and check which ones are bases. It is enough to find one base which can then be used to find the \emph{representation} of the delta-matroid as an $n$ by $2n$ matrix over $\mathbb Q^E\oplus\mathbb Q^{E^*}$. The linearly independent admissible $n$-sets of columns of the representation will correspond to the bases of the map \cite[theorem 4.3.5]{borovik_gelfand_white}. However, we shall not consider representations of delta-matroids in this paper.
\end{exmp}

We note that the definition \ref{definition: delta-matroid} can be modified so that a delta-matroid is specified by a collection of admissible $n$-sets \cite[section 4.1.2]{borovik_gelfand_white}. In that case we must replace $F_1$, $F_2$, $\mathcal F$, $x$, $y$ and $\{x,y\}$ with $B_1$, $B_2$, $\mathcal B$, $\{x,x^*\}$, $\{y,y^*\}$ and $\{x,x^*,y,y^*\}$, respectively. The reason that we chose our definition is due to the fact that if $\map$ is a map on the sphere, then its feasible sets correspond precisely to its spanning trees and therefore the delta-matroid in question is a matroid.

As in the case of matroids, there exists a notion of a dual delta-matroid.

\begin{prop}Let $\Delta(E)$ be a delta-matroid with $\mathcal F$ as its collection of feasible sets. Then the collection
\[\mathcal F^*=\{E\setminus F\mid F\in\mathcal F\}\] is the collection of feasible sets of some delta-matroid on $E$.
\end{prop}

This proposition is easily seen to be true by noting that \[F_1\symmdiff F_2=(E\setminus F_1)\symmdiff (E\setminus F_2).\] The delta-matroid on $E$ with $\mathcal F^*$ as the collection of its feasible sets is called the \emph{dual delta-matroid of} $\Delta(E)$ and is denoted by $\Delta^*(E)$.

\begin{thm}\label{dualthm}Let $\map$ be a map and $\mathcal B$ the collection of its bases. Let $\map^*$ be its dual map and $\Delta(\map^*)$ the delta-matroid of $\map^*$. Then $\Delta^*(\map) \cong \Delta(\map^*)$.
\end{thm}
\begin{proof}The bases of $\map$ and $\map^*$ clearly coincide. Therefore, the collection of feasible sets of $\Delta(\map^*)$ is
\[\mathcal F'=\{E^*\cap B\mid B\in\mathcal B\}.\]
If $F$ is a feasible set of $\Delta(\map)$, then $E\setminus F$ is a feasible set of $\Delta^*(\map)$, and we have
\begin{eqnarray*}
E\setminus F    & = & E\cap F^c = E\cap (B\cap E)^c\\
                & = & E\cap (B^c \cup E^*)=E\cap B^c\\
                & = & E\cap B^*,
\end{eqnarray*}
where $B^*$ is the admissible $n$-subset obtained from $B$ by starring and un-starring the un-starred and starred elements, respectively. Denote by $\psi\colon E\to E^*$ the bijection $\psi(i)=i^*$. From the computation above we have
\[\psi(E\setminus F)=\psi(E)\cap\psi(B^*)=E^*\cap B.\]
Hence $\Delta^*(\map)$ and $\Delta(\map^*)$ are isomorphic. Moreover, by relabelling the edges of $\map^*$ with the elements of $E$ we can even achieve equality between the two delta-matroids.
\end{proof}

If we recall that for plane maps the feasible sets correspond to spanning trees, we immediately recover theorem \ref{theorem: whitney}. In other words, a delta-matroid $\Delta(\map)$ is a matroid if, and only if $\map$ is a plane map.

\section{Galois action on the delta-matroids of maps}\label{section: Galois action on matroids}

Since delta-matroids do not take into account the bipartite structure of dessins, throughout this section we shall consider maps only. Nevertheless, this restriction is not a significant one, as established by the following corollary \cite[p.~50]{schneps94} to theorem \ref{belyi}.

\begin{cor}Let $X$ be an algebraic curve defined over $\mathbb{C}$. Then $X$ is defined over $\algbr$ if, and only if there is a clean \belyi map $f\colon X\to\rsphere$.
\end{cor}
This corollary is due to the fact that if $\vartheta\colon X\to\rsphere$ is a \belyi function, then $f=4\vartheta(1-\vartheta )$ is a clean \belyi function on the same curve $X$. The dessin to which it corresponds is a familiar one: it is the dessin obtained from $(X,\vartheta)$ by colouring all the white vertices black and adjoining the edge midpoints as the white vertices.

As we have seen, delta-matroids of maps are defined through a topological property, namely connectedness, and therefore we cannot expect that conjugate maps will have isomorphic delta-matroids. This indeed is the case, as we will see in the following examples.

\begin{exmp}\normalfont \label{exmp51}Let $A$, $B_+$ and $B_-$ be the three genus 0 clean dessins depicted in figure \ref{figure9} with \belyi functions
\[f(z)=16 \frac{(391+550\nu+455\nu^2)(z+2\nu)(z+1)^2z^5}{(16z-\nu+7\nu^2-4)(-8z+4\nu+3\nu^2-4)^2},\]
where $\nu$ is a root of the irreducible polynomial
\[7\nu^3+2\nu^2-\nu-4.\]
The dessin $A$ corresponds to its real root, while $B_+$ and $B_-$ correspond to its imaginary roots with positive and negative real parts, respectively \cite[fig. 87-89]{adrianov09}. Clearly, any two are conjugate.
\begin{figure}[ht]
  \centering
  \includegraphics[scale=.5,trim={7cm 15cm 7cm 4.5cm}]{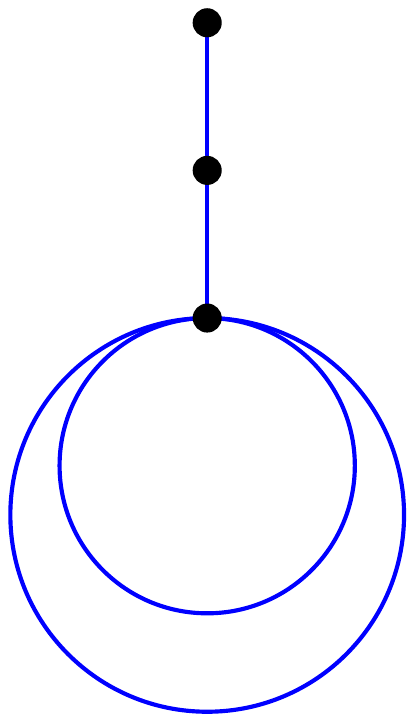}
  \includegraphics[scale=.5,trim={7cm 15cm 7cm 4.5cm}]{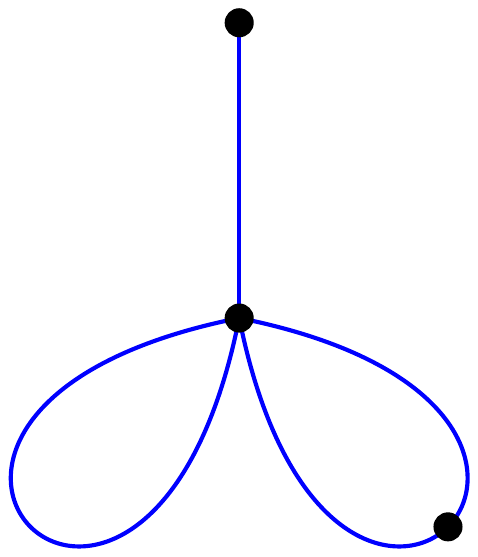}
  \includegraphics[scale=.5,trim={7cm 15cm 7cm 4.5cm}]{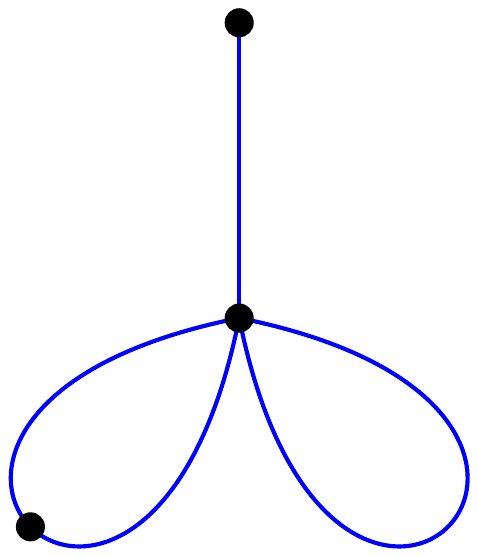}
  \caption{From left to right: dessins $A$, $B_+$ and $B_-$.}\label{figure9}
\end{figure}

Since these dessins are plane maps, their delta-matroids are matroids and the feasible sets are their spanning trees. Dessins $B_+$ and $B_-$ clearly have isomorphic delta-matroids with two feasible sets, while $A$ has only one feasible set.
\end{exmp}

\begin{exmp}\normalfont Let us look at some delta-matroids which are not matroids. Let $A_+$, $A_-$ and $B$ be the three genus 1 clean dessins as depicted and labelled in figure \ref{figure10}.
\begin{figure}[ht]
  \centering
  \includegraphics[scale=.4,trim={6cm 14cm 6cm 4.75cm}]{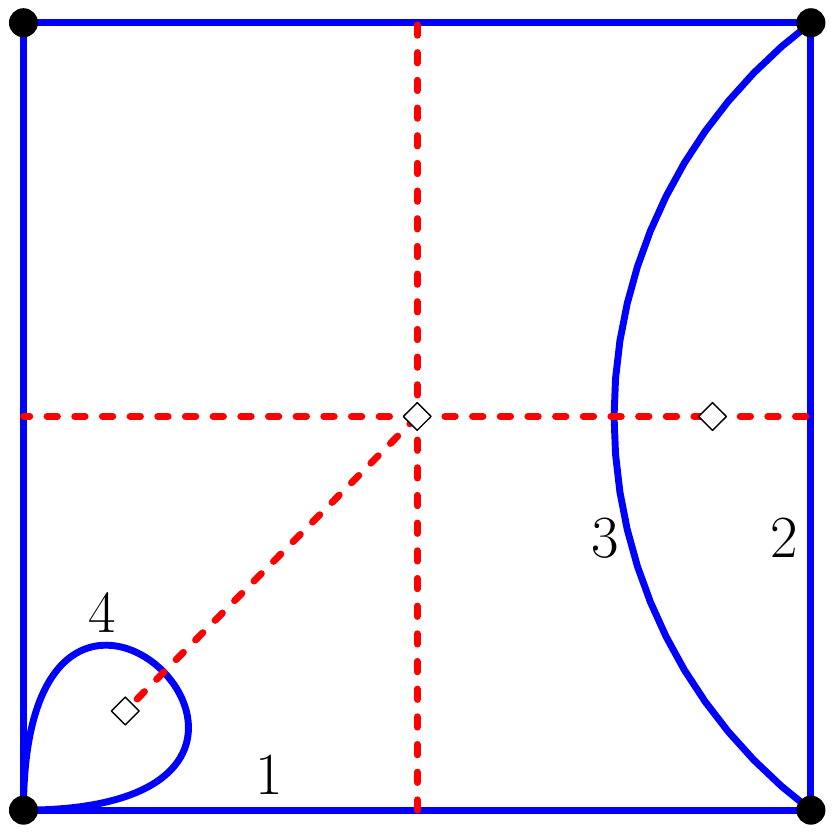}
  \includegraphics[scale=.4,trim={6cm 14cm 6cm 4.75cm}]{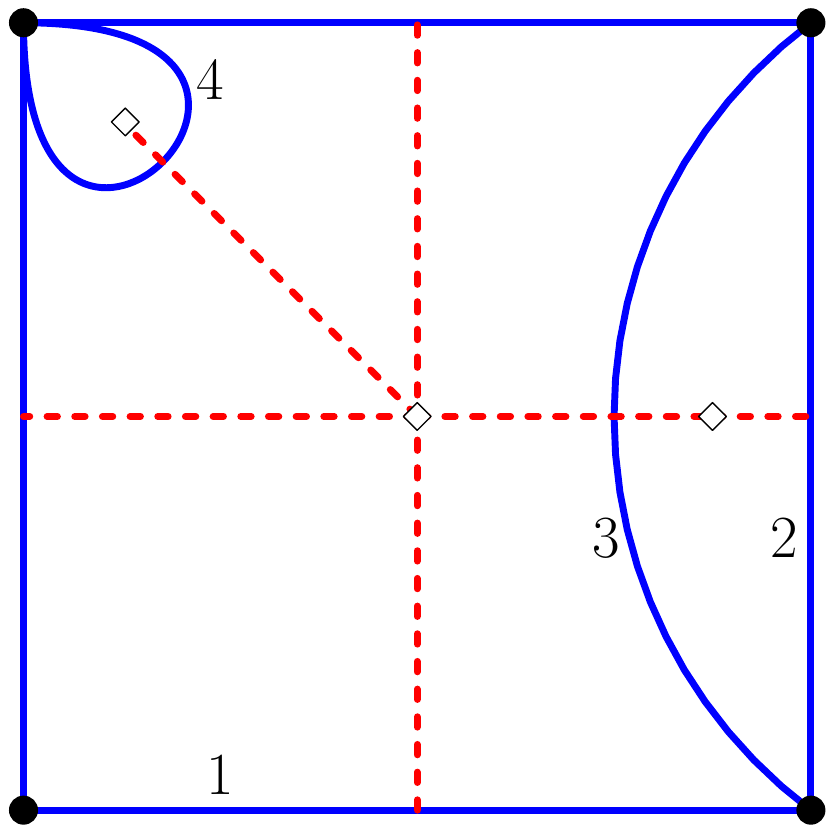}
  \includegraphics[scale=.4,trim={6cm 14cm 6cm 4.75cm}]{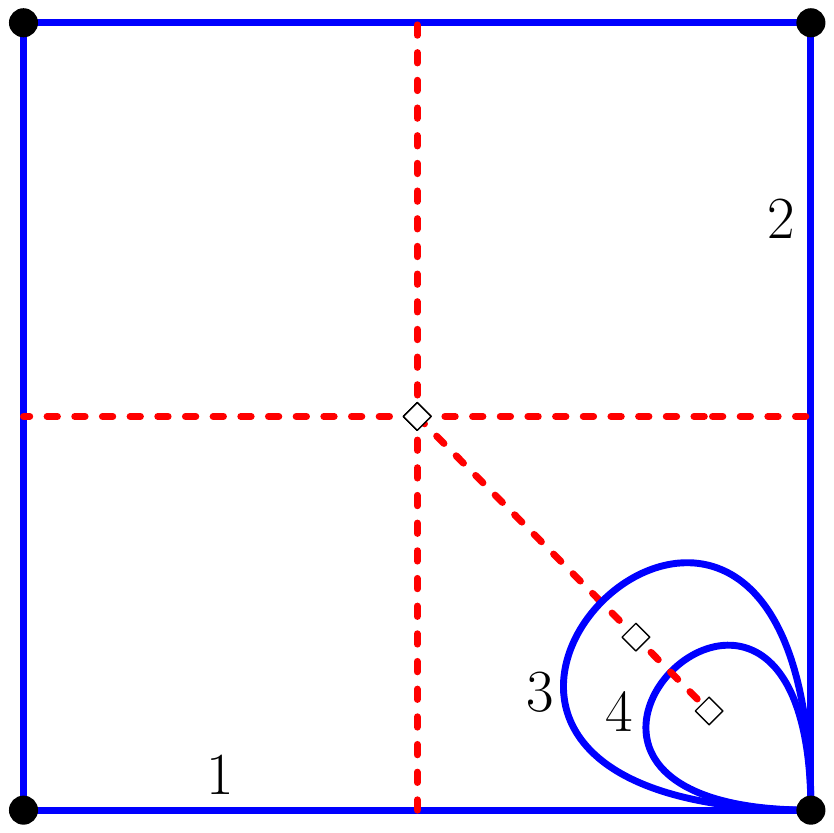}
  \caption{From left to right: dessins $A_+$, $A_-$ and $B$.}\label{figure10}
\end{figure}
The \belyi pairs of the three dessins have coefficients in the fixed field corresponding to the Galois group of the irreducible polynomial
\[256\nu^3-544\nu^2+1427\nu-172,\] and any two are conjugate. Similarly to the previous example, the dessin $B$ corresponds to the \belyi pair defined over $\mathbb R$ while the \belyi pairs for $A_+$ and $A_-$ are complex-conjugate. Due to the complicated expressions involved, we shall omit the equations for the \belyi pairs. However, the reader may look them up in \cite[pp.~39--40]{adrianov09}.

The bases of $A_+$ and $A_-$ are $123^*4^*$, $12^*34^*$ and $1^*2^*3^*4^*$ hence the feasible sets are $12$, $13$ and $\emptyset$. However, $B$ has only two bases, namely $123^*4^*$ and $1^*2^*3^*4^*$ and therefore has only two feasible sets: $12$ and $\emptyset$. The reason why delta-matroids fail to be Galois invariant is illustrated clearly in this example: a delta-matroid takes into account the topology of edges and hence distinguishes between non-contractible and contractible loops on the surface whereas $\absgal$ does not!
\end{exmp}

\subsection{Trivial delta-matroidal Galois invariants}\label{subsection: trivial}

The simplest dessins are the trees, that is genus 0 dessins with precisely one face. As we have already mentioned in the last paragraph before section \ref{subsection: invariants}, the action of $\absgal$ on the set of trees is very rich since it is faithful. However, delta-matroids associated to trees do not reveal much information as every tree has precisely one feasible set, the tree itself.

Similarly, $\absgal$ will preserve the delta-matroid of a genus 0 dessin which has $n$ faces of degree 1 and one face of arbitrary degree. Such a dessin is a tree with $m$ loops attached to it. Again, every such dessin clearly has only one feasible set, namely the tree obtained by removing the $m$ loops. Therefore, we have the following proposition.

\begin{prop}Let $D$ be a genus 0 clean dessin which is either
\begin{itemize}
  \item[(i)] a tree,
  \item[(ii)] a tree with $m$ degree 1 faces attached, or
  \item[(iii)] the dual dessin of a dessin of type $(i)$ or $(ii)$.
\end{itemize}
If $D'$ is a dessin conjugate to $D$, then $\Delta(D')\cong\Delta(D)$.
\end{prop}
\begin{proof}In the cases (i) and (ii) the proof is trivial if we recall that the passport of a dessin is a Galois invariant. Hence
the conjugate dessin $D'$ must be of the same type as $D$ in both cases. Since the delta-matroids of those dessins are one and the same feasible set, namely the (underlying) tree, we must have $\Delta(D')\cong\Delta(D)$.

For $(iii)$, recall from invariant \ref{duality} that the duals of conjugate dessins are conjugate as well. Since $D^*$ is of type (i) or $(ii)$ we have $\Delta(D'^*)\cong\Delta(D^*)$. Combining with theorem \ref{dualthm} we have
\[\Delta^*(D')=\Delta(D'^*)\cong \Delta(D^*) = \Delta^*(D).\]
Now by noting that $(\Delta^*)^*=\Delta$, we recover $\Delta(D')\cong\Delta(D)$.\end{proof}

As we have seen in example \ref{exmp51}, the case (ii) cannot be improved even to trees with only one degree 2 face attached. The following conjugate dessins found in \cite{wolfart06} show that case (i) cannot be extended to quasi-trees.

\begin{exmp}\normalfont Let $T_5$ denote the fifth Chebyshev polynomial of the first kind and consider its square
\[T_5^2(x)=25 x^2-200 x^4+560 x^6-640 x^8+256 x^{10}.\]
This polynomial is a clean \belyi map with critical points in the set
\[\left\{0,\frac{1\pm\sqrt 5}{4},\frac{-1\pm\sqrt 5}{4},\sqrt{\frac{5\pm\sqrt 5}{8}},\sqrt{\frac{-5\pm\sqrt 5}{8}}\right\}.\]
Therefore, if $X$ is the algebraic curve
\[y^2=(x-1)(x+1)\left(x-\sqrt{\frac{5+\sqrt 5}{8}}\right),\]
then the composition $t=T_5^2 \circ \pi_x$, where $\pi_x\colon X\to\rsphere$ is the projection to the first coordinate, is a clean \belyi map. Clearly $D=(X, t)$ will have precisely one face since $\inv t(\infty)=\{\infty\}$, as we can see in figure \ref{figure11}.
\begin{figure}[ht]
  \centering
  \includegraphics[scale=.7,trim={5cm 18cm 5cm 4.85cm}]{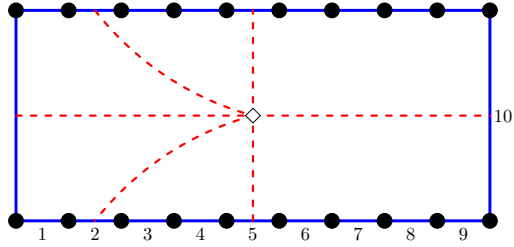}
  \caption{The dessin $(X,t)$. The only feasible set containing 10 is the entire dessin. Any two coedges $1\leq e,f\leq 9$ disconnect $X$ so other feasible sets must be of the form $1\cdots \hat e \cdots 9$, where $\hat e$ is omitted.}\label{figure11}
\end{figure}

Let $D$ be labelled as in figure \ref{figure11} and let $B$ be a base of $D$. If the edge 10 is in $B$ then no coedges can appear since cuts along the two edges 10 and $e^*$, for any $e\in\{1,\dots, 9\}$, will clearly disconnect $X$. Therefore, $B=12\cdots10$ is the only base containing the edge 10. On the other hand, if $10^*$ is in $B$ then at least one coedge $e^*\in\{1^*,\dots,9^*\}$ must appear since $1\cdots9(10)^*$ disconnects $D$. But if two or more coedges in $\{1^*,\dots, 9^*\}$ appear in $B$ then $D$ will again be disconnected. Therefore, $\Delta(D)$ has precisely 10 feasible sets, namely $12\cdots 10$ and $1\cdots \hat e\cdots 9$, where $\hat e$ denotes the omission of $e\in\{1,2,\dots 9\}$.

Now let $\theta$ be an automorphism in $\absgal$ such that
\[\theta\colon\sqrt{\frac{5+\sqrt 5}{8}}\mapsto\sqrt{\frac{5-\sqrt 5}{8}}.\]
Since $T_5^2$ is defined over the rationals, then $(T_5^2)^\theta$ coincides with $T_5^2$ and therefore $t^\theta$ and $t$ coincide as well. However, $X^\theta$, which is given by
\[y^2=(x-1)(x+1)\left(x-\sqrt{\frac{5-\sqrt 5}{8}}\right),\]
is a curve not isomorphic to $X$. Hence $D^\theta$ and $D$ are non-isomorphic conjugate dessins. The corresponding map is shown in figure \ref{figure12}.
\begin{figure}[ht]
  \centering
  \includegraphics[scale=.7,trim={5cm 18cm 5cm 4.85cm}]{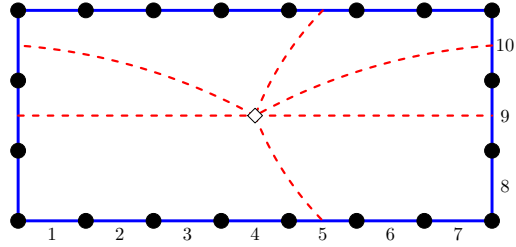}
  \caption{The dessin $D^\theta=(X^\theta,t)$. There are at least 18 feasible sets obtained by adjoining $1\cdots \hat e\cdots 7$, where $\hat e$ is omitted, to 89, 8(10) or 9(10).}\label{figure12}
\end{figure}

Let $D^\theta$ be labelled as in figure \ref{figure12} and $B$ a base of $D^\theta$. If the edges 8, 9 and 10 are in $B$, then $B$ must be the entire dessin. Now suppose that 8, 9 and $10^*$ are in $B$. Then the rest of $B$ must be of the form $1\cdots \hat e \cdots 7$, where $\hat e\in\{1,\dots 7\}$ is omitted. We can conclude the same for bases that contain $8$, $9^*$, $10$ or $8^*$, $9$, $10$. Therefore $\Delta (D^\theta)$ has at least 19 feasible sets and cannot be isomorphic to $\Delta(D)$.
\end{exmp}

\begin{qst}\normalfont As we have seen, $\absgal$ alters significantly the delta-matroids of conjugate dessins. In the cases where the delta-matroid is preserved, most information about the dessin is not captured. Is there an interesting family of dessins for which delta-matroids could provide some useful information?
\end{qst}

\subsection{Self-duality of maps and matroids}\label{subsection: self duality}
Recall that a map is self-dual if it is isomorphic to its dual. As an example, any map in figure \ref{figure9} is self-dual.

We say that a delta-matroid is \emph{self-dual} if $\Delta(E)\cong \Delta^*(E)$. Combining with theorem \ref{dualthm}, the delta-matroid of a map $\map$ is self-dual if, and only if $\Delta(\map)\cong \Delta(\map^*)$.

Self-dual maps clearly have self-dual delta-matroids. The following example demonstrates that the converse need not be true.

\begin{exmp}\normalfont Consider the map in figure \ref{figure13}. It is not self-dual since it has only one vertex of degree 1, while the dual map has two. However, both have precisely one feasible set corresponding to the unique spanning tree. Clearly their delta-matroids are isomorphic, as the two feasible sets are of the same size.
\begin{figure}[ht]
  \centering
  \includegraphics[scale=.5,trim={5cm 14.4cm 3cm 5.1cm}]{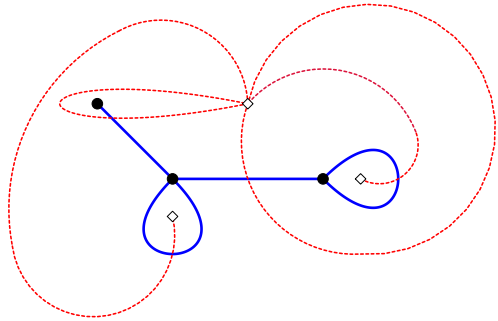}
  \caption{A map which is not self-dual but has a self-dual delta-matroid.}\label{figure13}
\end{figure}
\end{exmp}

By a theorem of Steinitz\footnote{Also, see 8.2.16 in \cite{oxley92}. There the same theorem is attributed to Whitney.} \cite[pp.~63]{mohar_thomassen01}, a 3-connected planar simple graph $G$ has, up to isomorphism, a unique embedding on the sphere. Moreover, if the delta-matroid of $G$ is self-dual, then $G$, as a planar map, is self-dual as well. Hence a 3-connected planar simple graph is self-dual as a map if, and only if its delta-matroid is self-dual. As we have mentioned in section \ref{subsection: invariants}, the property of being self-dual is a Galois invariant, and therefore the conjugates of 3-connected plane simple maps with self-dual delta-matroids must have a self-dual delta-matroid. Can the same be said, at least in the genus 0 case, for all clean dessins with self-dual delta-matroids? It is easy to see by inspecting the catalogue \cite{adrianov09} that this is the case for genus 0 dessins with 4 edges or less. However, this might be due to the simplicity of orbits involved; the largest orbit in the catalogue consists of only 3 dessins. Here we pose the following question.

\begin{qst}\normalfont \label{question: self-duality}Given a genus 0 clean dessin $D$, if the delta-matroid of $D$ is self-dual, does the same hold for any dessin conjugate to $D$?\end{qst}

Since in the genus 0 case the feasible sets of $D$ correspond to spanning trees, and if $v$ is the number of vertices, then any feasible set must have $v-1$ edges. Moreover, if $F$ is a feasible set of $D$, then $E\setminus F$ is a feasible set of $D^*$ and therefore $D$ must have $2v-2$ edges. Euler's formula now implies that the number $f$ of faces of $D$ has to be $f=v$. Therefore, if a counterexample is to be found, its passport should be of the following form
\[[a_1^{\alpha_1}\cdots a_j^{\alpha_j},2^{2v-2},b_1^{\beta_1}\cdots b_k^{\beta_k}],\]
with the following equalities satisfied:
\begin{eqnarray*}
\alpha_1+\cdots+\alpha_j &=& \beta_1+\cdots+\beta_k=v,\\
a_1\alpha_1+\cdots+a_j\alpha_j &=& b_1\beta_1+\cdots+b_k\beta_k=4v-4.
\end{eqnarray*}

In higher genus feasible sets are not all of the same size and therefore there are less constraints on the passport. This would suggest that a question analogous to question \ref{question: self-duality} is even less likely to have a positive answer.


\begin{qst}\normalfont Are there some other properties of delta-matroids that are invariant under the action of $\absgal$?
\end{qst}

\section{Partial duals and twists of delta-matroids}\label{section: partial duals}

A \emph{partial dual} of a map is a generalisation of the geometric dual of a map. It was first introduced in \cite{chmutov09} and later generalised to hypermaps in \cite{chmutov14}, where a representation as a triple of permutations is given as well. In this paper we shall first define partial duals combinatorially and then explain the geometric counterpart, thus working in the opposite direction of \cite{chmutov14}. We shall consider maps only but give some remarks on hypermaps as well. Throughout this section $D=\per$ will denote a clean dessin with $n$ edges, hence $\alpha$ will be of the form $\alpha=c_1\cdots c_n$, where $c_1, \dots, c_n$ are $n$ disjoint transpositions. We are identifying the edges of $D$ with the cycles of $\alpha$ so that the $j$-th edge corresponds to the transposition $c_j$. The notation $D/j$ stands for the map $D$ with the edge $j$ contracted, while $D\setminus j$ stands for the map $D$ with the edge $j$ deleted.

\begin{dfn}\normalfont Let $D=\per$ be a map. The \emph{partial dual with respect to an edge} $j$ of $D$ is the map
\[\partial_j D = (\sigma c_j, \alpha, c_j \varphi).\]
\end{dfn}

The following theorem shows that the partial dual with respect to an edge is well defined.

\begin{thm}\label{partial}Let $D=\per$ be a map. Then $\sigma c_j \alpha c_j \varphi=1$ and the group $\left\langle\sigma c_j, \alpha, c_j \varphi\right\rangle$ acts transitively on $\{1,\dots,2n\}$.
\end{thm}
\begin{proof}Since $c_j$ commutes with $\alpha$ we clearly have $\sigma c_j\alpha c_j \varphi=1$. If $n=1$ we are done since in that case $\partial_j D$ corresponds to the geometric dual of $D$. Hence suppose that $n>1$.

Without loss of generality set $c_j=(1~2)$ and let $a,b\in\{1,\dots,2n\}$. If $(a~b)$ is a cycle in $\alpha$, then $a^\alpha = b$ and we are done. Otherwise, let $\sigma_1$ and $\sigma_2$ (with possibly $\sigma_1=\sigma_2$) be the cycles of $\sigma$ corresponding to the (black) vertices of $D$ incident to the darts $1$ and $2$, respectively. Since we are assuming $n>1$, the two cycles $\sigma_1$ and $\sigma_2$ cannot both be trivial and neither can be equal to $c_j$.

We may assume that $a,b\notin\{1,2\}$ as well since if, say, $a=1$ and $\sigma_1$ is not trivial, then $a^{\sigma c_j}\notin\{1,2\}$. If $\sigma_1$ is trivial, then
    \[a^{(\sigma c_j)^2}=2^{\sigma c_j}.\]
Since $\sigma_2$ is not trivial, we clearly must have $2^{\sigma c_j}\not\in\{1,2\}$.

\begin{inparaenum}
	\item[\emph{Case (i)}.]Suppose that $\sigma_1$ and $\sigma_2$ are disjoint. Consider the not necessarily connected map $D\setminus
    j=\hat D\cup \tilde D$ obtained from $D$ by deleting the edge $j$. Let $\hat\sigma$, $\hat\alpha$ and $\tilde\sigma$, $\tilde\alpha$ be the restrictions of $\sigma$ and $\alpha$ on $\hat D$ and $\tilde D$, respectively. Clearly $\hat \sigma$  coincides with the restriction of $\sigma c_j$ on $\hat D$, and similarly $\tilde\sigma$ coincides with the restriction of $\sigma c_j$ on $\tilde D$.

    If $a$ and $b$ both belong to the same connected component, say $\hat D$, then there is $\hat g\in\left\langle \hat\sigma,\hat\alpha \right\rangle$
    such that $a^{\hat g}=b$. If $\hat g$ is of the form
    \[\hat g=\hat \sigma^{ v_1}\hat\alpha^{ w_1}\cdots\hat\sigma^{ v_k}\hat\alpha^{ w_k},\]
    and since on $\hat D$ we have $\hat\sigma=\sigma c_j$ and $\hat\alpha=\alpha$, then for
    \[g=(\sigma c_j)^{ v_1}\alpha^{ w_1}\cdots(\sigma c_j)^{ v_k}\alpha^{ w_k}\]
    we must have $a^g=b$ as well.

    If $a$ belongs to $\hat D$ and $b$ to $\tilde D$, then suppose that the vertex that corresponds to $\sigma_1$ in $D$ is in $\hat D$. Let $d$ be a dart in $\hat D$ such that in the map $D$ we have $d^\sigma=1$. By repeating the previous argument, there is $g\in\interval{\sigma c_j,\alpha}$ such that $a^g=d$.
    By acting with $\sigma c_j$ on $d$ twice we first map $d$ to 2 and then to some dart in $\tilde D$. Therefore, $a^{g(\sigma c_j)^2}$ and $b$ are now both in $\tilde D$. By reusing the same argument as before we can find $h\in\interval{\sigma c_j,\alpha}$ such that
    \[a^{g(\sigma c_j)^2 h}=b.\]

    \item[\emph{Case (ii)}.] Suppose that $\sigma_1$ and $\sigma_2$ coincide, that is
    \[\sigma_1=\sigma_2=(1~p_1~\cdots p_r~2~q_1\cdots~q_s).\]
    The product $\sigma_1c_j$ will split $\sigma_1$ into two cycles $\sigma'_1$ and $\sigma'_2$ such that
    \begin{eqnarray*}
        \sigma'_1 &=& (1~p_1~\cdots p_r),\\
        \sigma'_2 &=& (2~q_1\cdots~q_s).
    \end{eqnarray*}
    Let $D'$ be the not necessarily connected map obtained from $D$ by splitting the vertex corresponding to $\sigma_1=\sigma_2$ so that the orderings of the darts around the two new vertices correspond to $\sigma'_1$ and $\sigma'_2$. By connecting the new vertices with an edge with darts labeled by $\{2n+1, 2n+2\}$, a connected map with $\sigma'_1$ and $\sigma'_2$ disjoint is obtained. Now case \emph{(ii)} follows from~\emph{(i)} by noting that $D$ and $D'$ with the new edge $(2n+1 ~ 2n+2)$ contracted are equivalent maps.
\end{inparaenum}
\end{proof}

\begin{rmrk}\normalfont \label{remark51}When $D$ is a general dessin, i.e. a bipartite map (or equivalently, a hypermap), and $c_j$ a cycle in $\alpha$, then the partial dual with respect to the $j$-th white vertex (equivalently, $j$-th hyperedge) is the bipartite map
\[\partial_j D = (\sigma c_j, \inv{c_j}\hat\alpha,c_j\varphi),\]
where $\hat \alpha$ denotes the permutation obtained from $\alpha$ by omitting the cycle $c_j$.
\end{rmrk}

The geometric interpretation of the partial dual $\partial_j D$ for $c_j=(1~2)$ is the following. Suppose that $n>1$ and $c_j$ is not a loop. Let $\sigma_1$ and $\sigma_2$ be the two cycles of $\sigma$ which contain 1 and 2, respectively. Draw the dual edge $j^*$ of $j$ by crossing $j$ at the white vertex. The coedge $j^*$ is incident to at most two face centers marked with $\diamond$ as before; draw a segment joining a face center to a black vertex of $j$ if, and only if the black vertex is on the boundary of the corresponding face. As a result, four triangles are formed. Using the orientation of the underlying surface of $D$ shade the two triangles with vertices oriented as $\bullet - \circ - \diamond - \bullet$. Exactly one of those triangles has the dart 1 as its side. Label the $\circ-\diamond$ segment of that triangle with $1^*$, and proceed similarly with the other triangle. See figure \ref{figure14}.

\begin{figure}[ht]
  \centering
  \includegraphics[trim={7cm 19cm 7cm 4.65cm}]{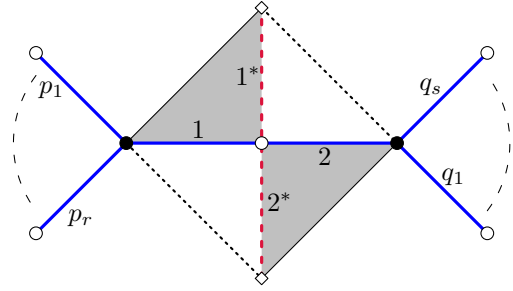}
  \caption{The darts of the coedge are labeled so that $i$ and $i^*$ are sides of the same shaded triangle, for $i=1,2$. Here $\sigma_1\sigma_2=(1~p_1\cdots~p_r)(2~q_1\cdots~q_s)$.}\label{figure14}
\end{figure}

Now contract $j$, and if $j^*$ is not already a loop, glue the endpoints of $j^*$ together and consider them as a single white vertex. If necessary, add a handle to the underlying surface of $D$ so that $(D/j)\cup\{j^*\}$ is a map. Then $\partial_j D$ is obtained by relabeling $j^*$, $1^*$ and $2^*$ into $j$, $1$ and $2$, respectively. The cycle corresponding to the new vertex is given by $\sigma_1\sigma_2c_j$. See figure \ref{figure15}.

\begin{figure}[ht]
  \centering
  \includegraphics[trim={4.85cm 18.5cm 4cm 4.5cm},scale=.85]{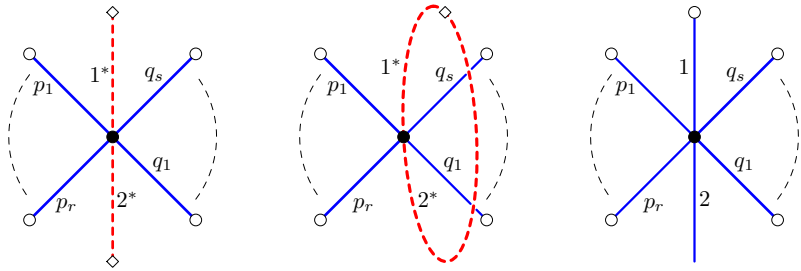}
  \caption{From left to right: contraction, then gluing of the endpoints and relabelling. By comparing with figure \ref{figure14} we see that $\sigma_1\sigma_2c_j=(1~p_1\cdots~p_r~2~q_1\cdots~q_s)$.}\label{figure15}
\end{figure}

If $c_j$ is a loop we proceed in the reverse direction. That is, first we break the loop at its white vertex so that the two endpoints fall onto some, possibly the same, face centers. If need be, remove a handle from the underlying surface. Then we split $\sigma_1=\sigma_2$ into two vertices and add an edge $j^*$ between them so that the former loop $j$ intersects it at its midpoint. Next we label the darts of $j^*$ as before. Finally, the partial dual is completed by deleting $j$ and relabeling $j^*$ to $j$ together with its darts. See figure \ref{figure16}.

\begin{figure}[ht]
  \centering
  \includegraphics[trim={8cm 12cm 5.2cm 4.5cm}, scale=.75]{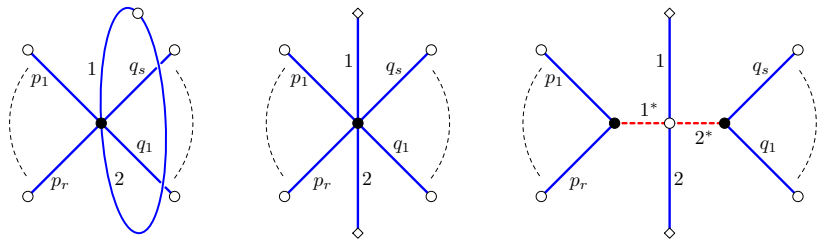}
  \caption{From left to right: a map with a cycle $\sigma_1 = \sigma_2 = (1~p_1\cdots~p_r~2~q_1\cdots~q_s)$. The loop is then broken at its white vertex and the two endpoints fall onto face centers. We split the vertex and add a new edge $j^*$. The final step is obtained by deleting $j$ and relabelling.  By comparing with figure \ref{figure15} we see that $\sigma_1c_j=\sigma_2c_j=(1~p_1\cdots~p_r)(2~q_1\cdots~q_s)$.}\label{figure16}
\end{figure}

\begin{exmp}\normalfont \label{exmp61}Let $D$ be the genus 0 dessin given by the triple
\[D=\big((1~4)(2~3), (1~2)(3~4), (1~3)(2~4)\big).\]
Let $c_1=(1 2)$. Then $\partial_1 D$ is the genus 1 dessin given by the triple
\[\partial_1 D=\big((1~4~2~3), (1~2)(3~4), (1~3~2~4)\big).\]
See figure \ref{figure17} for the geometric counterparts.
\begin{figure}[ht]
  \centering
  \includegraphics[trim={4.8cm 18.3cm 3cm 4.9cm}, scale=.75]{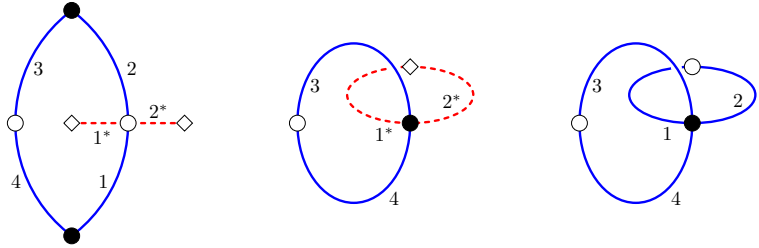}
  \caption{From left to right: the map $D$ from example \ref{exmp61}, an intermediate step, and its partial dual $\partial_1 D$.}\label{figure17}
\end{figure}
\end{exmp}

Since the cycles of $\alpha$ commute, the following is well defined.
\begin{dfn}\normalfont Let $D$ be a map, $E$ its set of edges and $S=\{i_1,\dots,i_k\}$ some subset of $E$. Then the partial dual of $D$ with respect to the set of edges $S$ is the map
\[\partial_S D=\partial_{i_k}\cdots\partial_{i_1}D=(\sigma c_{i_1}\cdots c_{i_k},\alpha, c_{i_k}\cdots c_{i_1}\varphi).\]
\end{dfn}

The geometric interpretation is immediately clear; the partial dual with respect to the set $S$ is obtained by dualising the edges in $S$ one at a time.

\begin{rmrk}\normalfont When $D$ is a general dessin, the partial dual with respect to some subset of hyperedges is obtained analogously to remark \ref{remark51}.
\end{rmrk}

The following lemma, borrowed directly from \cite{chmutov09,chmutov14}, lists some properties of the operation of partial duality.
\begin{lemma}\label{lemma: partial dual properties}Let $D$ be a map, $E$ its set of edges and $S$ some subset of $E$. Then
\begin{enumerate}[(a)]
  \item  $\partial_E D = D^*$
  \item  $\partial_S \partial_S D= D$.
  \item\label{lemma: partial dual properties: c}  If $j \in E\setminus S$, then $\partial_j\partial_S D =\partial_{S\cup\{j\}}D$.
  \item\label{lemma: partial dual properties: d}  If $S'$ is some other subset of $E$, then $\partial_{S'}\partial_S D=\partial_{S\symmdiff S'} D$.
  \item  Partial duality preserves orientability of hypermaps.
  \item\label{comment}  If $X$ is the underlying surface of $\partial_S D$, then $X$ is the underlying surface of $\partial_{E\setminus S} D$ as well.
\end{enumerate}
\end{lemma}

We shall comment only on part (\ref{comment}) of the lemma as other properties follow directly from the definition. For the partial dual $\partial_{E\setminus S} D$ we have
\[\partial_{E\setminus S} D=\partial_{E\symmdiff S} D=\partial_E \partial_S D.\]
Therefore, $\partial_{E\setminus S} D$ and $\partial_S D$ are dual maps and hence they are embedded on homeomorphic surfaces. Moreover, if $f$ is the clean \belyi function of $\partial_{E\setminus S} D$, then the two corresponding \belyi pairs are $(X,f)$ and $(X,1/f)$, respectively. Hence part (\ref{comment}) of the lemma can be improved slightly by noting that the underlying surfaces of $\partial_S D$ and $\partial_{E\setminus S} D$ coincide not just as topological, but as Riemann surfaces too.

\subsection{Partial duals, delta-matroids and the Galois action}\label{subsection: partial duals Galois}
Given a dessin $D=(X,f)$, the absolute Galois group acts on it and its partial duals. It appears that the relationship between the \belyi function of $D$ and $\partial_j D$ is very complicated. For if $D$ is a tree, its \belyi function is a polynomial; however, the \belyi function of $\partial_j D$, for any edge $j$, clearly is no longer polynomial. More worryingly, example \ref{exmp61} shows that the Riemann surface of $\partial_j D$ can be a point of a completely different moduli space than the one of $D$!

Nevertheless, some \emph{nice} behaviour can be observed. For example, we shall prove that $D$ always has a partial dual defined over its field of moduli by using a correspondence between delta-matroids and partial duals established in \cite[thm. 4.8]{moffat14}.

We start with a simple proposition.
\begin{prop}\label{obvious}Let $D=\per$ be a map, $E$ its set of edges and $S$ some subset of $E$. Then $\mon D$ is abelian if, and only if $\mon{\partial_S D}$ is abelian.\end{prop}
\begin{proof}By lemma \ref{lemma: partial dual properties} it is enough to consider $S=\{1\}$. Let $c_1$ be the corresponding cycle in $\alpha$. Then
\[\sigma\alpha=\alpha\sigma \iff \sigma\alpha c_1=\alpha\sigma c_1 \iff (\sigma c_1) \alpha=\alpha(\sigma c_1), \]
since $c_1$ commutes with $\alpha$.
\end{proof}

It was shown in \cite{hidalgo11}\footnote{For an alternative argument, see the discussion after proposition 3 in \cite{con_Jon_stre_wolf} as well.} than any dessin with abelian monodromy group is defined over $\mathbb Q$. Therefore the following corollary is obvious.

\begin{cor}Let $D=\per$ be a map such that $\mon D$ is abelian. Then $D$ and its partial duals are all defined over $\mathbb Q$.
\end{cor}

\begin{rmrk}\normalfont \label{rmrk53}Proposition \ref{obvious} is no longer true if $D$ is a hypermap. For if $c$ is a non-trivial cycle in $\alpha$ which is not a transposition, then $\inv c \hat\alpha = \inv c \alpha \inv c=c^{-2}\alpha$. Furthermore, if $\mon D$ is abelian we have
\begin{eqnarray*}
(\sigma c) (\inv c\hat\alpha) &=& (\inv c\hat\alpha)(\sigma c) \iff \\
\sigma\alpha c^{-1} &=& c^{-2}\alpha\sigma c \iff \\
\sigma \inv c \alpha &=& c^{-2}\sigma c \alpha \iff \\
c^2\sigma &=& \sigma c^2.
\end{eqnarray*}
The last equality does not hold always, of course. For example, if
\[D=\big((1~2)(3~4)(5~6),(1~3~5)(2~4~6),(1~6~3~2~5~4)\big)\]
is a dessin (see figure \ref{figure18}) then $\mon D \cong \mathbb Z_6$, however for $c=(1~3~5)$ we have $\sigma c^2 \neq c^2 \sigma$.
\begin{figure}[ht]
  \centering
  \includegraphics[trim={8cm 18cm 8cm 4.45cm}]{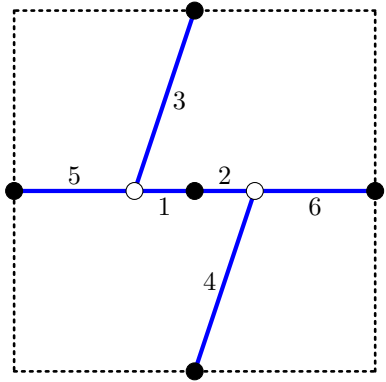}
  \caption{The dessin $D$ from remark \ref{rmrk53}.}\label{figure18}
\end{figure}
\end{rmrk}

Given a delta-matroid $\Delta (E)$ on some set $E$ with $\mathcal F$ as its collection of feasible sets, one can easily see that for some subset $S$ of $E$ the collection
$$\mathcal F \symmdiff S = \{F\symmdiff S\mid F\in\mathcal F\}$$
satisfies the symmetric axiom of definition \ref{definition: delta-matroid}. This motivates the following.

\begin{dfn}\normalfont Let $\Delta (E)$ be a delta-matroid on $E$ with $\mathcal F$ as its collection of feasible sets. Let $S$ be a subset of $E$. The delta-matroid on $E$ with $\mathcal F\symmdiff S$ as its collection of feasible sets is called \emph{the twist of} $\Delta (E)$ \emph{with respect to} $S$ and is denoted by $\Delta (E)*S$.
\end{dfn}
Similarly as before, when $D$ is a map, we shall use the notation $\Delta (D)*S$. The following lemma from \cite{moffat14} gives a correspondence between delta-matroids and partial duals.

\begin{lemma}\label{lemma: partial}Let $D$ be a map, $E$ its set of edges and $S$ some subset of $E$. Then
\[\Delta (\partial_S D)=\Delta (D)*S.\]
\end{lemma}
\begin{proof}It is sufficient to show the lemma for $S=\{j\}$ since the general result will then follow from lemma \ref{lemma: partial dual properties}(\ref{lemma: partial dual properties: c}).

If $j$ is in no base, then it is a contractible loop in $D$ and in $\partial_j D$ it is a pendant, i.e. an edge incident to a degree 1 vertex. In that case, the lemma follows easily.

So suppose that $B$ is a base of $D$ with $j\in B$. Moreover, suppose that $j$ is not a loop.
If $j$ is a pendant, then the lemma is again obvious. Therefore, suppose that both vertices incident to $j$ have degree at least 2.

By our construction, $j$ is a loop in $\partial_j D$. Therefore, $D/j$ is the same map as $(\partial_j D)\setminus j$. The underlying surface of $D/j$ is the surface of $D$, hence $B\setminus j$ does not disconnect it. Therefore, $B\setminus j$ is a base of $(\partial_j D)\setminus j$ as well.

Let us now adjoin the loop $j$ back to $(\partial_j D)\setminus j$. If we were forced to add a handle, then $j^*$ will not disconnect the underlying surface since it will split the new handle into two sleeves and leave the rest of the surface unaffected. Therefore, $(B\setminus j)\cup j^* = B\symmdiff \{j,j^*\}$ will be a base of $\partial_j D$. Furthermore, if $F$ is the feasible set of $\Delta (D)$ with $F=E\cap B$, then
\[F \symmdiff j = E\cap (B\symmdiff \{j,j^*\})\]
is a feasible set of $\Delta (\partial_j D)$.

If a new handle was not needed, then $\partial_j D$ and $(\partial_j D)\setminus j$ are on the same surface~$X$.
Since $(\partial_j D)\setminus j$ is a map on $X$ with at least one face, adjoining $j$ to it will clearly split some face into two new faces. Hence $j^*$ must be a contractible segment on $X$ since its endpoints are in the two faces with $j$ as a common boundary. Therefore, $B\symmdiff \{j,j^*\}$ is a base of $\partial_j D$ and, by passing to feasible sets, we conclude that $F\symmdiff j$ is a feasible set in $\Delta (\partial_j D)$, if $F$ is a feasible set in $\Delta (D)$.

Now suppose that $j$ is a loop. Since $j\in B$, it cannot be contractible. If $D$ and $\partial_j D$ are on the same surface, then, topologically, $j\in D$ and $j^*\in\partial_j D$ are the same loop. Therefore, $B\symmdiff \{j,j^*\}$ must be a base of $\partial_j D$.
Otherwise, by removing a handle, Euler's formula implies that $\partial_j D$ gained an additional face. By construction, $j$ must be on the boundary of the additional face, and at least one other face since other edges in $D$ do not contribute to the partial dual. Therefore, $j^*$ is contractible and $B\symmdiff \{j,j^*\}$ a base for $\partial_j D$.


So far we have shown that $\Delta (D)*j \subseteq \Delta (\partial_j D)$. The other inclusion is obtained by noting that if $F\in\Delta(\partial_j D)$, then
\[(F\symmdiff j) \in \Delta(\partial_j D)*j.\]
However, by using the just proven inclusion we have
\[(F\symmdiff j) \in \Delta (\partial_j\partial_j D)=\Delta (D).\]
Moreover, since $F=(F\symmdiff j)\symmdiff j$, we must have $F\in \Delta(D)*j$.\end{proof}

\begin{rmrk}\normalfont The proof of the preceding lemma is somewhat more natural in the language of ribbon graphs, as it can be seen in \cite[thm. 4.8]{moffat14}. However, in this paper, we prefer to work with maps instead.\end{rmrk}

We finish this section by demonstrating that partial duals with respect to feasible sets can be defined over their fields of moduli.

\begin{thm}\label{theorem: field of modul}Let $D$ be a clean dessin and $E$ its set of edges. Then $D$ has a partial dual which can be defined over its field of moduli.
\end{thm}
\begin{proof}
Recall that by theorem \ref{theorem: bachelor} a dessin can be defined over its field of moduli if it has a black vertex, or a white vertex, of a face center which is unique for its type and degree. If $D$ has precisely one face, then that face is the unique face of some degree and therefore both $\partial_\emptyset D=D$ and $\partial_E D=D^*$ can be defined over their corresponding fields of moduli (which coincide).

Otherwise, let $F\neq E$ be a feasible set of $\Delta (D)$ and set $S=E\setminus F$. Then by lemma \ref{lemma: partial} the map $\partial_S D$ has $S\symmdiff F =E$ as a feasible set. Therefore, $E$ is a base of $\partial_S D$. Furthermore, if $X_S$ is the underlying surface of $\partial_S D$, then $X_S\setminus \partial_S D$ is connected. This implies that $\partial_S D$ has precisely one face. As before, theorem \ref{theorem: bachelor} implies that $\partial_S D$ can be defined over its field of moduli.
\end{proof}

\begin{cor}Let $D$ be a clean dessin and $\Delta (D)$ its delta-matroid. If $F$ is a feasible set of $\Delta (D)$, then both $\partial_F D$ and $\partial_{E\setminus F} D$ can be defined over their fields of moduli. Moreover, the two fields coincide.
\end{cor}
\begin{proof}The case for $\partial_{E\setminus F} D$ was discussed in the proof the previous theorem. The second case follows from lemma \ref{lemma: partial dual properties} (\ref{lemma: partial dual properties: d}), that is
$$\partial_E (\partial_{E\setminus F} D)=\partial_F D.$$
Since the fields of definition of a map and its dual map coincide, and both maps can be defined over their field of moduli, then the fields of moduli coincide as well.
\end{proof}

\section{Maps, their partial duals and tropical curves}\label{section: tropical}
In this section we informally comment on a simple relationship between the monodromy groups of dessins, partial duals and tropical curves. To the best knowledge of the author, this relationship has not been noted in the literature yet. We do not assume any knowledge of tropical geometry, however the reader is referred to \cite{maclagan12} for an introduction.

Let $D=\per$ be a clean dessin with
\[\sigma = v_1\cdots v_j,\ \alpha=c_1\cdots c_n,\ \varphi=f_1\cdots f_k,\]
and consider the planar graph $G$ obtained from the triple $\per$ in the following way.
    \begin{itemize}
        \item Mark the integer points in the segment $[0,n+1]$.
        \item Place $j$ vertices, one for each cycle in $\sigma$, vertically above 0.
        \item To a vertex $i$ attach an open segment of length 1 and label it with the cycle~$v_i$.
        \item Choose a cycle $(p ~ q)$ in $\alpha$.
            \begin{itemize}
                \item If $p$ and $q$ are in the cycles $v_p$ and $v_q$, respectively, above 1 join the edges with labels $v_p$ and $v_q$ into a single edge of length 1/2, so that a degree 3 vertex above 1 is formed. Label the edge with the cycle $\sigma_p\sigma_q (p~q)$.
                \item If $p$ and $q$ are in the same cycle, say $v_r$,
                    above 1 split the edge with label $v_r$ into two edges of length 1/2, so that a degree 3 vertex above 1 is formed. Label the two edges with the cycles in $\sigma_r (p~ q)$.
                \item Extend all other edges so that their ends are above 3/2.
            \end{itemize}
        \item Repeat the previous step until all the cycles of $\alpha$ are exhausted. Above $n+1$ there are $k$ vertices, one for each cycle of $\varphi$. The edges incident with the final vertices have labels corresponding to the cycles in $\inv\varphi$.
    \end{itemize}
Planar graphs obtained in this fashion are called \emph{monodromy graphs} \cite{cav_john_mark10, johnson12}. Let us look at an example.
\begin{exmp}\normalfont \label{example: monodromy graph 1}Let $D=\per$ be the map $B$ from figure \ref{figure10}. It can be represented by the triple
\[\big((1~3~5~7~8~6~2~4),(1~2)(3~4)(5~6)(7~8),(1~6~3~2~4)(5~8)(7)\big).\]
Therefore, above 0 we should have one vertex, and above 5 we should have three vertices. A monodromy graph obtained by multiplying $\sigma$ in the order $(1~2)$, $(3~4)$, $(5~6)$ and $(7~8)$ is given in figure \ref{figure19}.
\begin{figure}[ht]
  \centering
  \includegraphics[trim={5.65cm 18.5cm 3cm 4.5cm},scale=.75]{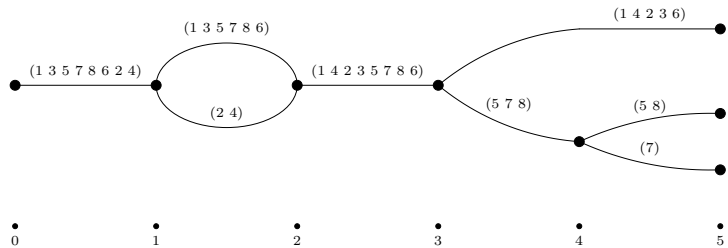}
  \caption{A monodromy graph for the map $D=\per$ from example \ref{example: monodromy graph 1}.}\label{figure19}
\end{figure}

Multiplying $\sigma$ with the cycles of $\alpha$ in a different order may produce a different monodromy graph. For example, if we multiply in the order $(1~2)$, $(5~6)$, $(3~4)$, $(7~8)$, the resulting monodromy graph shown in figure \ref{figure20} will not be isomorphic to the previous one since it will have a cycle of length 3.
\begin{figure}[ht]
  \centering
  \includegraphics[trim={5.65cm 18cm 3cm 4.8cm},scale=.75]{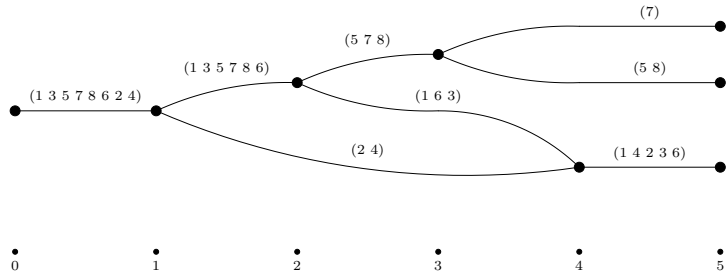}
  \caption{A monodromy graph for the map $D=\per$ from example \ref{example: monodromy graph 1} not isomorphic to the monodromy graph in figure \ref{figure19}.}\label{figure20}
\end{figure}
\end{exmp}

Irregardless of the order in which we multiply the cycles of $\sigma$ with the cycles of $\alpha$, monodromy graphs capture all of the information contained in the passport of a clean dessin $D$. Clearly the number and the degrees of black vertices and face centers correspond to the number of vertices and the lengths of the labels of edges above $0$ and $n+1$, and the genus of $D$ corresponds to the genus of the graph, which is defined as the first Betti number of the graph (this fact is a simple consequence of the handshaking lemma). Moreover, the vertices of the graphs correspond precisely to the partial duals of $D$ and two trivalent vertices $v$ and $w$ are adjacent if, and only if $\partial_j v=w$ or $\partial_j w=v$ for some edge $j$. Furthermore, monodromy graphs transfer dessins into the realm of tropical geometry.

\begin{dfn}\normalfont An \emph{abstract tropical curve} is a connected graph without vertices of degree 2 and with edges decorated by the elements of the set $\left\langle 0,\infty\right]$. The decorations on the edges are called \emph{lengths}. Edges incident to degree 1 vertices have length $\infty$ and all other edges have finite length.
\end{dfn}

It is easy to see how to pass from a clean dessin $D$ to an abstract tropical curve: first form a monodromy graph for $D$ and decorate each edge with the length of its corresponding cycle. Finally, decorate the edges incident to degree 1 vertices with~$\infty$. Tropical curves obtained in this way capture most information contained in the passport, and since they depend only on the monodromy group of the dessin, the following is clear:
\begin{thm}Let $D$ and $D'$ be clean dessins and $\mathcal T$ and $\mathcal T'$ the sets of abstract tropical curves obtained from the monodromy graphs of $D$ and $D'$, respectively. If $D$ and $D'$ are conjugate, then any two curves $T\in \mathcal T$ and $T'\in\mathcal T'$ have
\begin{itemize}
  \item the same number of finite edges and the same number of infinite edges.
  \item The same number of degree 3 vertices.
  \item The same \emph{genus}, which is defined as the genus of the underlying monodromy graph. In particular, if $D\simeq D'$ is a tree, then $T$ and $T'$ are tropical trees.
\end{itemize}
\end{thm}


The invariants above most likely do not improve on the already known invariants. However, they may serve as a motivation for studying tropical curves in the context of the theory of dessins d'enfants.

\subsection*{Acknowledgements}The author would like to thank the organisers and participants of the SIGMAP14 conference\footnote{5th Workshop SIGMAP - Symmetries In Graph, Maps And Polytopes.
Sponsored by the Open University, London Mathematical Society and British Combinatorial Committee. Dates: 7\textsuperscript{th} - 11\textsuperscript{th} July 2014. Location: ELIM Conference Centre, West Malvern, U.K. \url{http://mcs.open.ac.uk/SIGMAP/}} for the opportunity to present a talk on which this work is based, and for the lovely and informative presentations and discussions throughout. The author is also grateful for the invaluable guidance provided by his PhD advisor Prof. Alexandre Borovik.

%
%
%

\end{document}